\documentclass[11pt]{article}

\usepackage{tikz}
\usepackage{subfigure}
\usepackage{scalefnt}
\usepackage{amsmath,amsfonts,amsthm}
\usepackage{enumerate}

 \newtheorem{defn}{{\bf Definition}}[section]
 \newtheorem{eg}[defn]{{\bf Example}}
 \newtheorem{lemma}[defn]{{\bf Lemma}}
 \newtheorem{prop}[defn]{{\bf Proposition}}
 \newtheorem{theo}[defn]{{\bf Theorem}}
 \newtheorem{cor}[defn]{{\bf Corollary}}
 \newtheorem{remark}[defn]{{\bf Remark}}
 \newtheorem{conj}[defn]{{\bf Conjecture}}

\newcommand{\TPSS}{S^{\hspace{.2mm}3} \mbox{$\times
\hspace{-2.8mm}_{-}$} \, S^{\hspace{.1mm}1}}

\newcommand{\TPSSD}{S^{\hspace{.2mm}d-1} \mbox{$\times
\hspace{-2.8mm}_{-}$} \, S^{\hspace{.1mm}1}}

 \newcommand{\TPPSS}{\kern.24em \rule width.08em height1.5ex
depth-.08ex \kern-.36em \times}

\newcommand{\lk}[2]{{\rm lk}_{#1}(#2)} \newcommand{\st}[2]{{\rm
st}_{#1}(#2)}
\newcommand{\Kd}{{\mathcal{K}}(d)}
\newcommand{\lKd}{\overline{\mathcal{K}}(d)}

\begin{document}

\title{An infinite family of tight triangulations of manifolds}
\author{Basudeb Datta and Nitin Singh}

\date{}

\maketitle

\vspace{-10mm}
\begin{center}

\noindent {\small Department of Mathematics, Indian Institute of Science, Bangalore 560\,012, India.$^1$}


\footnotetext[1]{{\em E-mail addresses:} dattab@math.iisc.ernet.in (B. Datta), nitin@math.iisc.ernet.in (N.
Singh).}

\medskip

\date{Revised: June 11, 2013}
\end{center}


\hrule

\begin{abstract}
We give explicit construction of vertex-transitive tight triangulations of $d$-manifolds for $d\geq 2$. More
explicitly, for each $d\geq 2$, we construct two $(d^{\hspace{.3mm}2}+5d+5)$-vertex neighborly triangulated
$d$-manifolds whose vertex-links are stacked spheres. The only other non-trivial series of such tight
triangulated manifolds currently known is the series of non-simply connected triangulated $d$-manifolds with
$2d+3$ vertices constructed by K\"{u}hnel.  The manifolds we construct are strongly minimal. For $d\geq 3$, they
are also tight neighborly as defined by Lutz, Sulanke and Swartz. Like K\"{u}hnel's complexes, our manifolds are
orientable in even dimensions and non-orientable in odd dimensions.
\end{abstract}


\noindent {\small {\em MSC 2000\,:} 57Q15, 57R05.


}

\medskip

\hrule

\section{Introduction}

In \cite{wa}, Walkup introduced the class $\Kd$, $d\geq 2$, of simplicial complexes whose vertex-links are
stacked $(d-1)$-spheres. So, a member of Walkup's class $\Kd$ is a triangulated $d$-manifold for $d\geq 2$ and
any triangulated $2$-manifold is a member of ${\mathcal K}(2)$. The following result by Kalai \cite{ka} shows
that the members of this class triangulate a very natural class of manifolds obtained by handle additions on a
sphere.

\begin{prop}[Kalai]\label{prop:kalai}
For $d\geq 4$, a connected simplicial complex $X$ is in ${\mathcal K}(d)$ if and only if $X$ is obtained from a
stacked $d$-sphere by $\beta_1(X)$ combinatorial handle additions. In consequence, any such $X$ triangulates
either $(S^{\hspace{.1mm}d -1}\!\times S^1)^{\# \beta_1}$  or $(\TPSSD)^{\#\beta_1}$ according to whether $X$ is
orientable or not. $($Here $\beta_1 = \beta_1(X) = \beta_1(X;\mathbb{Z}_2).)$
\end{prop}

Walkup's class $\Kd$ has also been a major source of examples of tight triangulations. Recall that, for a field
$\mathbb{F}$, a $d$-dimensional simplicial complex $X$ is called {\em tight with respect to $\mathbb{F}$} (or
{\em $\mathbb{F}$-tight}) if {\rm (i)} $X$ is connected, and {\rm (ii)} for all induced sub-complexes $Y$ of $X$
and for all $0\leq j\leq d$, the morphism $H_j(Y;\mathbb{F})\rightarrow H_j(X;\mathbb{F})$ induced by the
inclusion map $Y\hookrightarrow X$ is injective \cite{ku95, bd16}. In this paper, by tight we mean tight with
respect to the field $\mathbb{Z}_2$.

Very few examples of tight triangulations are known. Apart from the trivial $(d+2)$-vertex triangulation
$S^{\hspace{.1mm}d}_{d+2}$ of the $d$-sphere $S^{\hspace{.1mm}d}$, the only non-trivial series of such
triangulations currently known is the $(2d+3)$-vertex non-simply connected triangulated manifolds
$K^{\hspace{.1mm}d}_{2d+3}$ constructed by K\"{u}hnel \cite{ku86}. The complex $K^{\hspace{.1mm}d}_{2d+3}$
triangulates an $S^{\hspace{.1mm}d-1}$-bundle over $S^1$. Not surprisingly, K\"{u}hnel's triangulations are
members of $\Kd$. Walkup's class also relates to one of the few combinatorial criteria for tightness that are
known (for more general combinatorial criteria see \cite[Theorem 3.10]{bd16}). For example, Effenberger \cite{ef}
showed that:

\begin{prop}[Effenberger]\label{prop:ef}
For $d\neq 3$, the neighborly members of $\Kd$ are tight.
\end{prop}

Analogous to Walkup's class $\Kd$, let $\overline{{\mathcal K}}(d)$ be the class of all simplicial complexes
whose vertex-links are stacked $(d-1)$-balls. So, a member of $\lKd$ is a triangulated $d$-manifold with
boundary. Recently, Bagchi and Datta \cite{bd17} proved

\begin{prop}[Bagchi and Datta]\label{prop:tight3}
If $M$ is a neighborly member of ${\mathcal K}(3)$ then the following are equivalent. {\rm (i)} $M$ is tight,
{\rm (ii)} $M$ is the boundary of a neighborly member of $\overline{{\mathcal K}}(4)$, and {\rm (iii)}
$\beta_1(M; \mathbb{Z}_2)= (f_0(M)-4)(f_0(M)- 5)/20$.
\end{prop}

Walkup's class is also closely related to the notion of tight
neighborly triangulation as introduced by Lutz, Sulanke and Swartz
in \cite{lss}. In particular, we have the following:

\begin{prop}[Novik and Swartz]\label{prop:novik-swartz}
Let $X$ be a connected triangulated $d$-manifold.  \begin{enumerate} \item[$(a)$] If $d\geq 3$ then
$\binom{d+2}{2}\beta_1(X; \mathbb{Z}_2) \leq f_1(X)-(d+1)f_0(X)+\binom{d+2}{2}$. \item[$(b)$] Further, if \,
$\binom{d+2}{2}\beta_1(X; \mathbb{Z}_2) = f_1(X)-(d+1)f_0(X)+\binom{d+2}{2}$ and $d\geq 4$ then $X\in {\mathcal
K}(d)$.
\end{enumerate}
\end{prop}

From Propositions \ref{prop:novik-swartz} and \ref{prop:kalai}, one can deduce the following\,:

\begin{cor}[Lutz, Sulanke and Swartz] \label{prop:lss}
Let $X$ be a connected triangulated $d$-manifold. If $d\geq 3$,  then \vspace{-2mm}
\begin{align}
\binom{d+2}{2}\beta_1(X; \mathbb{Z}_2) \leq \binom{f_0(X)-d-1}{2}.
\label{eq:tnbrly}
\end{align}
Moreover for $d\geq 4$, the equality holds if and only if $X$ is a neighborly member of ${\mathcal K}(d)$.
\end{cor}

For $d\geq 3$, a triangulated $d$-manifold is called {\em tight neighborly} if it satisfies (\ref{eq:tnbrly})
with equality.

In this paper, we present the second infinite series of neighborly members of $\Kd$ after K\"{u}hnel's series
$K^{\hspace{.1mm}d}_{2d+3}$. Like K\"{u}hnel's complexes, our manifolds also exhibit vertex-transitive
automorphism groups. They are orientable in even dimensions, non-orientable in odd dimensions. In view of the
above results, it follows that the triangulated $d$-manifolds we construct are tight for $d\geq 2$ and are tight
neighborly for $d\geq 3$. Our examples are also strongly minimal. More explicitly we have

\begin{theo}\label{thm:main}
For $d\geq 2$ and $n=d^{\hspace{.3mm}2}+5d+5$, there exist $n$-vertex non-isomorphic members $M^d_{n}$ and
$N^d_n$ of ${\mathcal K}(d)$ with the following properties.
\begin{enumerate}[{\rm (a)}]
 \item $M^d_{n}$ and $N^d_{n}$ are neighborly for all $d$.
 \item $M^d_{n}$ and $N^d_{n}$ are tight for all $d$.
 \item $M^d_{n}$ and $N^d_{n}$ are tight neighborly for $d\geq 3$.
 \item  $\beta_1(M^d_{n}; \mathbb{Z}_2) = \beta_1(N^d_{n}; \mathbb{Z}_2) =\binom{n-d-1}{2}/\binom{d+2}{2}=
 d^{\hspace{.3mm}2}+5d+6$ for $d\geq 3$.
 \item If $d\geq 2$ is even then $M^d_{n}$ and $N^d_{n}$ triangulate $(S^{\hspace{.1mm}d -1}\!\times
S^1)^{\# \beta}$ and if $d\geq 3$ is odd then $M^d_{n}$ and $N^d_{n}$ triangulate $(\TPSSD)^{\#\beta}$, where
$\beta = d^{\hspace{.3mm}2}+5d+6$.
 \item  $M^d_{n}$ and $N^d_{n}$ are strongly minimal for all $d$.
 \item  $\mathbb{Z}_n$ acts vertex-transitively on $M^d_{n}$ and $N^d_{n}$, respectively for all $d$.
\end{enumerate}
\end{theo}

For $d\geq 3$, apart from the $(d+2)$-vertex standard spheres $S^{\hspace{.1mm}d}_{d+2}$, K\"{u}hnel's complexes
$K^{\hspace{.1mm}d}_{2d+3}$ and a few sporadic examples, our examples $M^d_{d^2+5d+5}$ and $N^d_{d^2+5d+5}$ are
the only known tight neighborly triangulated manifolds (cf. Table 1 in Section 6). For $d\geq 3$,
$K^{\hspace{.1mm}d}_{2d+3}$ is the unique $(2d+3)$-vertex triangulated manifold with $\beta_1 \neq 0$ \cite{bd8,
css}. We pose the following.

\begin{conj} \label{conj}
For $d\geq 3$, if $X$ is a $(d^{\hspace{.3mm}2}+5d+5)$-vertex triangulated $d$-manifold with $\beta_1(X;
\mathbb{Z}_2) = d^{\hspace{.3mm}2}+5d+6$ then $X$ is isomorphic to $M^d_{d^{\hspace{.1mm}2}+5d+5}$ or
$N^d_{d^{\hspace{.1mm}2} +5d+5}$.
\end{conj}

The remainder of the paper is organized as follows. In Section \ref{sec:pre}, we review some basic definitions
and results. Explicit description of the manifolds in Theorem \ref{thm:main} appears in Section \ref{sec:exam}.
In Section \ref{sec:cons}, we present a purely combinatorial way of constructing neighborly members of $\lKd$ and
use it to construct the families $M^d_{d^{\hspace{.1mm}2}+5d+5}$ and $N^d_{d^{\hspace{.1mm}2} +5d+5}$. In Section
\ref{sec:proof}, we prove properties of the aforementioned manifolds mentioned in Theorem \ref{thm:main}.

\section{Preliminaries} \label{sec:pre}

All simplicial complexes considered here are finite and abstract. We identify two complexes if they are
isomorphic. By a triangulated manifold, sphere or ball, we mean a simplicial complex whose geometric carrier is a
topological manifold, sphere or ball, respectively.

A $d$-dimensional simplicial complex is called {\em pure} if all its maximal faces (called {\em facets}) are
$d$-dimensional. A $d$-dimensional pure simplicial complex is said to be a {\em weak pseudomanifold} if each of
its $(d - 1)$-faces is in at most two facets. For a $d$-dimensional weak pseudomanifold $X$, the {\em boundary}
$\partial X$ of $X$ is the pure subcomplex of $X$ whose facets are those $(d-1)$-dimensional faces of $X$ which
are contained in unique facets of $X$. The {\em dual graph} $\Lambda(X)$ of a pure simplicial complex $X$ is the
graph whose vertices are the facets of $X$, where two facets are adjacent in $\Lambda(X)$ if they intersect in a
face of codimension one. A {\em pseudomanifold} is a weak pseudomanifold with a connected dual graph. All
connected triangulated manifolds are necessarily pseudomanifolds.

If $X$ is a $d$-dimensional simplicial complex then, for $0\leq j \leq d$, the number of its $j$-faces is denoted
by $f_j = f_j(X)$. The vector $(f_0, \dots, f_d)$ is called the {\em face vector} of $X$ and the number $\chi(X)
:= \sum_{i=0}^{d} (-1)^i f_i$ is called the {\em Euler characteristic} of $X$. As is well known, $\chi(X)$ is a
topological invariant, i.e., it depends only on the homeomorphic type of $|X|$ and, for any field $\mathbb{F}$,
$\chi(X)= \sum_{i=0}^{d} (-1)^i \beta_i(X; \mathbb{F})$, where $\beta_i(X; \mathbb{F}) = \dim_{\mathbb{F}}(H_i(X;
\mathbb{F}))$ is the {\em $i$-th Betti number} of $X$ with respect to the field $\mathbb{F}$. A simplicial
complex $X$ is said to be {\em $l$-neighbourly} if any $l$ vertices of $X$ form a face of $X$. By a {\em
neighborly} complex, we shall mean a $2$-neighborly complex.

Let $X$ be a weak pseudomanifold with disjoint facets $\gamma$, $\delta$
and let $\psi \colon \gamma\to \delta$ be a bijection.
Let $X^{\psi}$
denote the weak pseudomanifold obtained from $X \setminus \{\gamma, \delta\}$ by identifying $x$ with $\psi(x)$
for each $x\in \gamma$. Then $X^{\psi}$ is said to be obtained from $X$
by a {\em combinatorial handle addition}.
If $u$ and $\psi(u)$ have no common neighbor in $X$ for each
$u\in \gamma$ (such a $\psi$ is called an {\em admissible} map) and
$X$ is in ${\mathcal K}(d)$ then $X^{\psi}$ is also in ${\mathcal K}(d)$
(see \cite{bd9}).

A {\em standard $d$-ball} is a pure $d$-dimensional simplicial complex with one facet. The standard ball with
facet $\sigma$ is denoted by $\overline{\sigma}$. A {\em standard $d$-sphere} is a simplicial complex isomorphic
to the boundary complex of a standard $(d+1)$-ball. The standard $d$-sphere on the vertex-set $V$ is denoted by
$S^{\hspace{.15mm}d}_{d+2}(V)$ (or simply by $S^{\hspace{.15mm}d}_{d+2}$). A simplicial complex $X$ is called a
{\em stacked $d$-ball} if there exists a sequence $B_1, \dots, B_m$ of simplicial complexes such that $B_1$ is a
standard $d$-ball, $B_m=X$ and, for $2\leq i\leq m$, $B_i = B_{i-1}\cup \overline{\sigma}_i$ and $B_{i-1} \cap
\overline{\sigma}_i = \overline{\tau}_i$, where $\sigma_i$ is a $d$-face of $B_i$ and $\tau_i$ is a $(d-1)$-face
of $\sigma_i$. Clearly, a stacked ball is a pseudomanifold. A simplicial complex is called a {\em stacked
$d$-sphere} if it is (isomorphic to) the boundary of a stacked $(d+1)$-ball. A trivial induction on $m$ shows
that a stacked $d$-ball actually triangulates a topological $d$-ball, and hence a stacked $d$-sphere is a
triangulated $d$-sphere. If $X$ is a stacked ball then clearly $\Lambda(X)$ is a tree. So, the dual graph of a
stacked ball is a tree. But, the converse is not true (e.g., the 7-vertex 3-pseudomanifold $P$ whose facets are
$1234, 2345, 3456, 4567, 1567$ is a pseudomanifold for which the dual graph $\Lambda(P)$ is a tree but $P$ is not
a triangulated ball). We have (\cite{bdns1})

\begin{lemma}\label{lemma:stackedball} Let $X$ be a pure simplicial complex of dimension $d$.
\begin{enumerate} \item[{\rm (i)}] If the dual graph $\Lambda(X)$ is a tree then $f_0(X) \leq f_d(X) +d$.
\item[{\rm (ii)}] The graph $\Lambda(X)$ is a tree and $f_0(X) = f_d(X) +d$ if and only if $X$ is a stacked ball.
\end{enumerate}
\end{lemma}

\begin{proof} Let $f_d(X)=m$ and $f_0(X)=n$. So, $\Lambda(X)$ is a graph with $m$ vertices. We prove (i) by
induction on $m$. If $m=1$ then the result is true with equality. So, assume that $m >1$ and the result is true
for smaller values of $m$. Since $\Lambda(X)$ is a tree, it has a vertex $\sigma$ of degree one (leaf) and hence
$\Lambda(X)-\sigma$ is again a tree. Let $Y$ be the pure simplicial complex (of dimension $d$) whose facets are
those of $X$ other than $\sigma$. Since $\sigma$ has a $(d-1)$-face in $Y$, it follows that $f_0(Y) \geq n-1$.
Since $f_d(Y) = m -1$, the result is true for $Y$ and hence $f_0(Y) \leq (m -1) +d$. Therefore, $n \leq f_0(Y) +1
\leq 1 + (m-1) +d = m+d$. This proves (i).

If $X$ is a stacked $d$-ball with $m$ facets then $X$ is a pseudomanifold and by the definition (since at each of
the $m-1$ stages one adds one facet and one vertex), $n = (d+1) + (m-1) = m+d$. Conversely, let $\Lambda(X)$ be a
tree and $n= f_0(X) = m+d$. Let $Y$, $\sigma$ be as above. Since $f_0(Y) \geq n -1$, it follows that $f_0(Y) = n$
or $n -1$. If $f_0(Y) = n$ then $f_0(Y) = n > (m-1) +d = f_d(Y) +m$, a contradiction to part (i). So, $f_0(Y) =
n-1$ and hence $Y\cap \overline{\sigma}$ is a $(d-1)$-face of $\sigma$. Since $f_d(Y) = m-1$, by induction
hypothesis, $Y$ is a stacked $d$-ball and hence $X = Y \cup \overline{\sigma}$ is a stacked $d$-ball. This proves
(ii).
\end{proof}

\begin{cor}\label{cor:C0}
Let $X$ be a pure $d$-dimensional simplicial complex and let $CX$ denote a {\em cone} over $X$. Then $CX$ is a
stacked $(d+1)$-ball if and only if $X$ is a stacked $d$-ball.
\end{cor}

\begin{proof} Notice that $f_{d+1}(CX)=f_d(X)$ and
$f_0(CX)=f_0(X)+1$. Also $\Lambda(CX)$ is naturally isomorphic to $\Lambda(X)$. The proof now follows from Lemma
\ref{lemma:stackedball}.
\end{proof}

Clearly, if $N \in \overline{{\mathcal K}}(d)$ then $N$ is a triangulated manifold with boundary and satisfies
\begin{align} \label{eq:skel}
{\rm skel}_{d-2}(N) = {\rm skel}_{d-2}(\partial N).
\end{align}
Here ${\rm skel}_{j}(N) := \{\alpha\in N : \dim(\alpha) \leq j\}$ is the $j$-skeleton of $N$. From \cite[Remark
2.20]{bd17}, it follows\,:

\begin{prop}[Bagchi and Datta]\label{prop:bagchidatta}
For $d\geq 4$, the map $M \mapsto \partial M$ is a bijection between $\overline{\mathcal K}(d+1)$ and $\Kd$.
\end{prop}

The following corollary follows from Proposition \ref{prop:bagchidatta} (cf. \cite{bdns1}).

\begin{cor}\label{cor:auto}
For $d\geq 4$, if $M\in \overline{\mathcal K}(d+1)$ then ${\rm Aut}(M)={\rm Aut}(\partial M)$.
\end{cor}

Note that any automorphism $\varphi$ of a pure simplicial complex $X$ induces an automorphism $\bar{\varphi}$ of
the dual graph $\Lambda(X)$ given by $\sigma \mapsto \varphi(\sigma)$ for any facet $\sigma$ of $X$. Here we
have\,:

\begin{lemma}\label{lemma:dualgraph}
Let $X$ be a pseudomanifold which is not a cone $($i.e., not all the facets are through a single vertex$)$. Then,
$\varphi \mapsto \bar{\varphi}$ is an injective group homomorphism from ${\rm Aut}(X)$ into ${\rm
Aut}(\Lambda(X))$. Thus, ${\rm Aut}(X)$ is naturally isomorphic to a subgroup of ${\rm Aut}(\Lambda(X))$.
\end{lemma}

\begin{proof}
Clearly, $\varphi \mapsto \bar{\varphi}$ is a group homomorphism. Let $\varphi$ be such that $\bar{\varphi}$ is
identity on $\Lambda(X)$. Thus $\varphi(\sigma) = \sigma$ for each facet $\sigma$ in $X$. Let $x\in V(X)$ be
arbitrary. Choose facets $\alpha,\beta$ such that $x\in \alpha$ and $x\not\in \beta$. As $\Lambda(X)$ is
connected, there is a path $\alpha_0\alpha_1\cdots \alpha_k$ in $\Lambda(X)$ with $\alpha_0=\alpha$ and $\alpha_k
= \beta$. Since $x\in \alpha_0$ and $x\not\in \alpha_k$, there exists $l<k$ such that $x$ is in $\alpha_0,
\alpha_1, \ldots, \alpha_l$ and $x\not\in \alpha_{l+1}$. Hence $\alpha_l\setminus \alpha_{l+1} = \{x\}$.  Now
$\varphi(\alpha_l) = \alpha_l$ and $\varphi(\alpha_{l+1})=\alpha_{l+1}$ imply $\varphi(x) =x$. Since $x$ was
arbitrary, we see that $\varphi$ is identity on $X$.
\end{proof}

A $d$-dimensional simplicial complex $X$ is called {\em minimal} if $f_0(X) \leq f_0(Y)$ for every triangulation
$Y$ of the geometric carrier $|X|$ of $X$. We say that $X$ is {\em strongly minimal} if $f_i(X) \leq f_i(Y)$,
$0\leq i \leq d$, for all such $Y$. In  \cite{bd16}, Bagchi and Datta have shown the following.

\begin{prop}[Bagchi and Datta]\label{prop:minimal}
For any field $\mathbb{F}$, each $\mathbb{F}$-tight member of\,  ${\mathcal K}(d)$ is strongly minimal.
\end{prop}

\section{Examples} \label{sec:exam}

In this section, we present our examples of neighborly members of $\Kd$ for every $d\geq 2$.

\begin{eg}\label{example:Md}
{\rm Let $d\geq 2$ and $n = d^{\hspace{.3mm}2}+5d+5$. Consider the $(d+1)$-dimensional pure simplicial complex
${\mathcal M}^{\hspace{.1mm}d+1}_{n}$ on the vertex set $\{a_0, a_1, \ldots, a_{n-1}\}$ whose $(d+2)n$ facets
are
\begin{align} 
\sigma_i  = & \{a_{i-j} \, : \, 0\leq j\leq d+1\}, \, \, \mu_i =\{a_i\}\cup \{a_{i+j(d+3)-1} \, : \, 1\leq j
\leq d+1\}, \nonumber \\
\alpha_{k,i}  = & \{a_i\}\cup \{a_{i-j} \, : \, 2\leq j\leq d+2-k\}\cup \{a_{i+j(d+3)-1}: 1\leq j\leq k\},
\end{align}
$0\leq i\leq n-1$, $1\leq k\leq d$. The subscripts (except the first subscript on $\alpha$) are to be taken
modulo $n$. For all $d\geq 2$, ${\mathcal M}^{\hspace{.1mm}d+1}_{n}$ is a neighborly member of
$\overline{\mathcal K}(d+1)$ (see Lemma \ref{lemma:Md}). We further define
\begin{equation}\label{eq:series-Md}
M^{\hspace{.1mm}d}_{n} := \partial {\mathcal M}^{\hspace{.1mm}d+1}_{n}.
\end{equation}
Since ${\mathcal M}^{\hspace{.1mm}d+1}_{n}\in \overline{{\mathcal K}}(d+1)$, we have $M^{\hspace{.1mm}d}_{n}\in
{\mathcal K}(d)$. By (\ref{eq:skel}), ${\rm skel}_{d- 1}(M^{\hspace{.1mm}d}_{n}) = {\rm skel}_{d-1}({\mathcal
M}^{\hspace{.1mm}d+ 1}_{n})$. This implies that $f_0(M^{\hspace{.1mm}d}_{n}) = f_{0}({\mathcal M}^{\hspace{.1mm}d
+1}_{n})=n$ and, since $d\geq 2$, $M^{\hspace{.1mm}d}_{n}$ is neighborly. }
\end{eg}

\begin{eg}\label{example:Nd}
{\rm Let $d\geq 2$ and $n = d^{\hspace{.3mm}2}+5d+5$. Consider the $(d+1)$-dimensional pure simplicial complex
${\mathcal N}^{\hspace{.1mm}d+1}_{n}$ on the vertex set $\{a_0, a_1, \ldots, a_{n-1}\}$ whose $(d+2)n$ facets
are
\begin{align} \label{eq:facets-Nd}
\sigma_i = & \, \{a_{i-j} \, : \, 0\leq j \leq d+1\}, \, \,
\mu_i = \{a_{i-j(d+3)} \, : \, 0\leq j\leq d+1\}, \nonumber \\
\alpha_{k,i} = & \, \{a_i\}\cup \{a_{i-j} \, : \, 2\leq j\leq d+2-k\} \cup \{a_{i-j(d+3)} \, : \, 2 \leq j \leq
k+1\},
\end{align}
$0\leq i\leq n-1$, $1\leq k\leq d$. The subscripts (except the first subscript on $\alpha$) are to be taken
modulo $n$. For all $d\geq 2$, ${\mathcal N}^{\hspace{.1mm}d+1}_{n}$ is a neighborly member of
$\overline{\mathcal K}(d+1)$ (see Lemma \ref{lemma:Nd}). We further define
\begin{equation}\label{eq:series-Nd}
N^{\hspace{.1mm}d}_{n} := \partial {\mathcal N}^{\hspace{.1mm}d+1}_{n}.
\end{equation}
Since ${\mathcal N}^{\hspace{.1mm}d+1}_{n}\in \overline{{\mathcal K}}(d+1)$,
we have $N^{\hspace{.1mm}d}_{n}\in {\mathcal
K}(d)$. By the similar arguments as in the case of  $M^{\hspace{.1mm}d}_{n}$, $N^{\hspace{.1mm}d}_{n}$ has $n$
vertices and is neighborly. }
\end{eg}

From the definition of ${\mathcal M}^{\hspace{.1mm}d+1}_{n}$ (resp., ${\mathcal N}^{\hspace{.1mm}d+1}_{n}$), the
permutation
\begin{align} \psi:= (a_0, a_1, \dots, a_{n-1}) \label{eq:psi}
\end{align}
is an automorphism of ${\mathcal M}^{\hspace{.1mm}d+1}_{n}$ (resp., ${\mathcal N}^{\hspace{.1mm}d+1}_{n}$). Since
any automorphism of ${\mathcal M}^{\hspace{.1mm}d+1}_{n}$ is an automorphism of $\partial {\mathcal
M}^{\hspace{.1mm}d+1}_{n}$, it follows that $\psi \in {\rm Aut}(M^{\hspace{.1mm}d}_{n})$. Similarly, $\psi \in
{\rm Aut}({\mathcal N}^{\hspace{.1mm}d+1}_{n}) \subseteq {\rm Aut}(N^{\hspace{.1mm}d}_{n})$.  Since the order of
$\psi$ is $n$, we get

\begin{lemma}\label{lemma:aut(MdNd1)}
$\mathbb{Z}_n$ acts vertex-transitively on ${\mathcal M}^{\hspace{.1mm}d+1}_{n}$, ${\mathcal N}^{\hspace{.1mm}d +
1}_{n}$, $M^{\hspace{.1mm}d}_{n}$ and $N^{\hspace{.1mm}d}_{n}$, respectively.
\end{lemma}

Observe that the induced automorphism $\bar{\psi}$ of $\Lambda({\mathcal M}^{\hspace{.1mm}d+1}_{n})$ (resp.,
$\Lambda({\mathcal N}^{\hspace{.1mm}d+1}_{n})$) is given by
\begin{align}
\bar{\psi} = ({\sigma}_0, \ldots, {\sigma}_{n-1})({\mu}_0, \ldots, {\mu}_{n-1}) \prod_{k=1}^{d}({\alpha}_{k,0},
\ldots, {\alpha}_{k,n-1}). \nonumber
\end{align}

We remark that triangulations of surfaces with cyclic automorphism group were also constructed by Ringel and
Youngs as part of their proof of the Map Color Theorem \cite[Chap 2,\ Sec 2.3]{Ringel}. As part of a series of
neighborly triangulations on $12s+7$ vertices, they obtained a neighborly triangulation of an orientable surface
on $19$ vertices with $\mathbb{Z}_{19}$ action. For $d=2$, the 19-vertex triangulated 2-manifold $M^2_{19}$
(resp., $N^2_{19}$) is obtained as the boundary of the triangulated 3-manifold ${\mathcal M}^3_{19}$ (resp.,
${\mathcal N}^3_{19}$). Our examples also exhibit $\mathbb{Z}_{19}$ 
action and are different
(non-isomorphic) from the one obtained in \cite{Ringel}. In the terminology of \cite[Chap 2,\  Sec 2.3]{Ringel},
the triangulations $R$ (Ringel), $M^2_{19}$ and $N^2_{19}$ are described by following cyclic permutations as
``row 0" (we identify the vertex $a_i$ with $i$).
\begin{align*}
R & : 1\  11\  14\  13\  15\  3\  8\  9\  7\  4\  17\  10\  18\  5\  16\ 12\  2\  6, \nonumber \\
M^2_{19} & : 1\  7\  3\  2\  11\  6\  18\  16\  4\  14\  8\  10\  15\  12\ 13\  5\  9\  17, \nonumber \\
N^2_{19} & : 1\  12\  3\  2\  6\  11\  18\  16\  9\  5\  13\  15\  10\  7\ 8\  14\  4\  17.
\end{align*}

Consider the 3-dimensional example $M^3_{29}$. Figure \ref{fig:lk_M_29} shows the link ${\rm lk}_{M^3_{29}}(a_0)$
of the vertex $a_0$ in $M^3_{29}$. Clearly, $\lk{M^3_{29}}{a_0}$ is a 28-vertex triangulation of the 2-sphere
$S^{\hspace{.15mm}2}$. By construction, we know that $\mathbb{Z}_{29}$ acts vertex-transitively on $M^3_{29}$.
These imply that the link of each vertex is a triangulated 2-sphere and hence $M^3_{29}$ is a triangulated
$3$-manifold. Here we prove

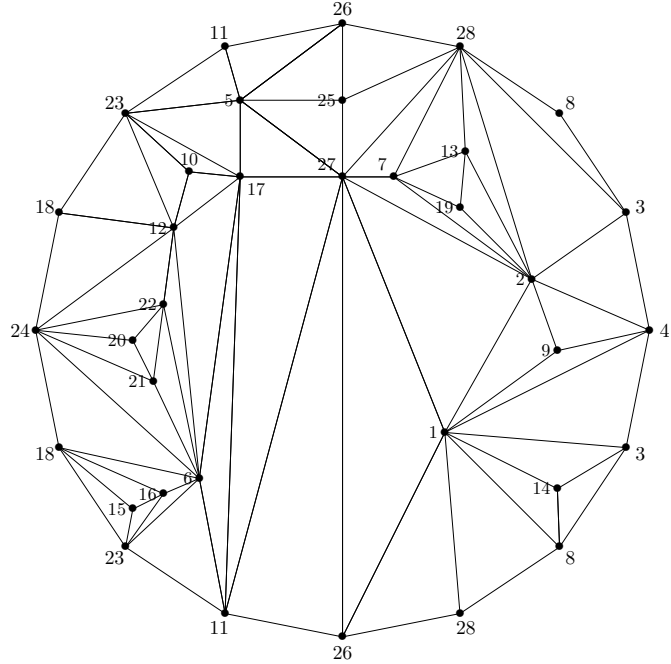
\begin{figure}[ht]
\centering \scalebox{0.68}{
\begin{tikzpicture}
\pgfmathsetmacro{\angle}{22.5};
\foreach \i/\j in {
    0/26,
    1/28,
    2/8,
    3/3,
    4/4,
    5/3,
    6/8,
    7/28,
    8/26,
    9/11,
    10/23,
    11/18,
    12/24,
    13/18,
    14/23,
    15/11
} {
    \draw (90-\angle*\i:6) node {$\bullet$};
    \draw (90-\angle*\i:6.3) node {$\j$};
}
\coordinate (V17) at (-2,3);
\coordinate (V27) at (0,3);
\coordinate (V7) at (1,3);
\coordinate (V10) at (-3,3.1);
\coordinate (V5) at (-2,4.5);
\coordinate (V12) at (-3.3,2);
\coordinate (V22) at (-3.5,0.5);
\coordinate (V21) at (-3.7,-1);
\coordinate (V6) at (-2.8,-2.9);
\coordinate (V13) at (2.4,3.5);
\coordinate (V19) at (2.3,2.4);
\coordinate (V2) at (3.7,1.0);
\coordinate (V9) at (4.2,-0.4);
\coordinate (V14) at (4.2,-3.1);
\coordinate (V1) at (2.0,-2.0);
\coordinate (V20) at (-4.1,-0.2);
\coordinate (V16) at (-3.5,-3.2);
\coordinate (V15) at (-4.1,-3.5);
\coordinate (V26) at (90:6);
\coordinate (V25) at (90:4.5);
\coordinate (V28) at (90-\angle:6);
\coordinate (V8) at (90-2*\angle:6);
\coordinate (V3) at (90-3*\angle:6);
\coordinate (V4) at (90-4*\angle:6);
\coordinate (V3b) at (90-5*\angle:6);
\coordinate (V8b) at (90-6*\angle:6);
\coordinate (V28b) at (90-7*\angle:6);
\coordinate (V26b) at (90-8*\angle:6);
\coordinate (V11b) at (90-9*\angle:6);
\coordinate (V23b) at (90-10*\angle:6);
\coordinate (V18b) at (90-11*\angle:6);
\coordinate (V24) at (90-12*\angle:6);
\coordinate (V18) at (90-13*\angle:6);
\coordinate (V23) at (90-14*\angle:6);
\coordinate (V11) at (90-15*\angle:6);

\draw (V17) node {$\bullet$};
\draw (V27) node {$\bullet$};
\draw (V7) node {$\bullet$};
\draw (V10) node {$\bullet$};
\draw (V5) node {$\bullet$};
\draw (V12) node {$\bullet$};
\draw (V22) node {$\bullet$};
\draw (V21) node {$\bullet$};
\draw (V6) node {$\bullet$};
\draw (V13) node {$\bullet$};
\draw (V19) node {$\bullet$};
\draw (V2) node {$\bullet$};
\draw (V9) node {$\bullet$};
\draw (V14) node {$\bullet$};
\draw (V20) node {$\bullet$};
\draw (V16) node {$\bullet$};
\draw (V15) node {$\bullet$};
\draw (V1) node {$\bullet$};
\draw (V25) node {$\bullet$};

\draw (V17) node[below right] {\small $17$};
\draw (V27) node[above left] {\small $27$};
\draw (V7) node[above left]  {\small $7$};
\draw (V10) node[above] {\small $10$};
\draw (V5) node[left] {\small $5$};
\draw (V12) node[left] {\small $12$};
\draw (V22) node[left] {\small $22$};
\draw (V21) node[left] {\small $21$};
\draw (V6) node[left] {\small $6$};
\draw (V13) node[left] {\small $13$};
\draw (V19) node[left] {\small $19$};
\draw (V2) node[left] {\small $2$};
\draw (V9) node[left] {\small $9$};
\draw (V14) node[left] {\small $14$};
\draw (V20) node[left] {\small $20$};
\draw (V16) node[left] {\small $16$};
\draw (V15) node[left] {\small $15$};
\draw (V1) node[left] {\small $1$};
\draw (V25) node[left] {\small $25$};

\draw (V27) -- (V2) -- (V9) -- (V4);
\draw (V3b) -- (V14) -- (V8b);
\draw (V26b) -- (V27);

\draw (V27) -- (V7) -- (V19) -- (V13) -- (V28) -- (V3);
\draw (V9) -- (V1) -- (V27);

\draw (V17) -- (V11b);
\draw (V23b) -- (V16) -- (V18b);
\draw (V24) -- (V21) -- (V22) -- (V12) -- (V17);

\draw (V11) -- (V5) -- (V17) -- (V10) -- (V12) -- (V18);
\draw (V18b) -- (V15) -- (V16) -- (V6) -- (V11b);

\draw (V17) -- (V5) -- (V26);
\draw (V28) -- (V7) -- (V2) -- (V1) -- (V26b);
\draw (V11b) -- (V17);

\draw (V3) -- (V2) -- (V13) -- (V7) -- (V27) -- (V5) -- (V26);
\draw (V26b) -- (V1) -- (V8b);

\draw (V26) -- (V5) -- (V23);
\draw (V23b) -- (V6) -- (V17) -- (V27) -- (V26b);

\draw (V10) -- (V23) -- (V5) -- (V27) -- (V11b) -- (V6) -- (V12) -- (V10);

\draw (V23) -- (V10) -- (V17) -- (V6) -- (V22) -- (V24);

\draw (V4) -- (V2) -- (V28);
\draw (V8b) -- (V14) -- (V1) -- (V4);

\draw (V26) -- (V25) -- (V27) -- (V17) -- (V23);

\draw (V24) -- (V6) (V15)--(V23b) (V23) -- (V12) -- (V24);

\draw (V18) -- (V12) -- (V22) -- (V20) -- (V21) -- (V6) -- (V18b);

\draw (V28) -- (V27) -- (V5) -- (V11);
\draw (V11b) -- (V27) -- (V1) -- (V28b);

\draw (V25) -- (V28) (V25) -- (V5)
      (V2) -- (V19)
      (V1) -- (V3b)
      (V20) -- (V24);

\draw (V26) -- (V28) -- (V8) -- (V3) -- (V4) -- (V3b) -- (V8b)
      -- (V28b) -- (V26b) -- (V11b) -- (V23b) -- (V18b) -- (V24)
      -- (V18) -- (V23) -- (V11) -- cycle;

\end{tikzpicture}
} \vspace{-2mm} \caption{Link of vertex $a_0$ in $M^{\hspace{.1mm}3}_{29}$ ($i$ stands for $a_i$)}
\label{fig:lk_M_29}
\end{figure}

\begin{lemma}\label{lemma:aut(M^3_29)}
The full automorphism group
of $M^{\hspace{.1mm}3}_{29}$ $($resp., $N^{\hspace{.1mm}3}_{29})$ is isomorphic to
$\mathbb{Z}_{29}$.
\end{lemma}

\begin{proof}
We present a proof for $M^{\hspace{.2mm}3}_{29}$. Similar arguments work for $N^{\hspace{.2mm}3}_{29}$.

Since $\mathbb{Z}_{29}$ acts vertex-transitively, it is sufficient to show that the stabilizer of the vertex
$a_0$ in ${\rm Aut}(M^{\hspace{.1mm}3}_{29})$ is the trivial subgroup.

Let $\beta\in {\rm Aut}(M^{\hspace{.1mm}3}_{29})$ and $\beta(a_0) =a_0$. So, $\beta\in {\rm
Aut}(\lk{M^3_{29}}{a_0})$. For $1\leq i\leq 28$, let $\deg(a_i)$ denote the number of edges through $a_i$ in
$\lk{M^3_{29}}{a_0}$. Clearly, $\deg(\beta(a_i)) = \deg(a_i)$. Since $\deg(a_i) = 7$ for $i= 12$ and 17,
$\beta(\{a_{12}, a_{17}\}) = \{a_{12}, a_{17}\}$ and hence $\beta(\{a_6, a_{10}\}) = \{a_{6}, a_{10}\}$. This
implies that $\beta(a_6) = a_6$ and hence $\beta(\lk{\lk{M^3_{29}}{a_0}}{a_6}) = \lk{\lk{M^3_{29}}{a_0}}{a_6}$.
Since $\lk{\lk{M^3_{29}}{a_0}}{a_6}$ is the 9-cycle $a_{11}a_{17}a_{12}a_{22}a_{21}a_{24}a_{18}a_{16}a_{23}
a_{11}$, $\deg(a_{22})=5$ and $\deg(a_{11}) =6$, it follows that $\beta$ is identity on
$\lk{\lk{M^3_{29}}{a_0}}{a_6}$. Then $\beta(\sigma) = \sigma$ for all simplices $\sigma$ in $\lk{M^3_{29}}{a_0}$
containing the vertex $a_6$. Since $\lk{M^3_{29}}{a_0}$ is a pseudomanifold, this implies that $\beta$ is the
identity on $\lk{M^3_{29}}{a_0}$. This proves that ${\rm Aut}(M^3_{29})$ is isomorphic to $\mathbb{Z}_{29}$.
\end{proof}

For the geometric carriers of $M^{\hspace{.1mm}3}_{29}$ and $N^{\hspace{.1mm}3}_{29}$, we have

\begin{lemma}\label{lemma:|(M^3_29|}
The simplicial complex $M^{\hspace{.1mm}3}_{29}$ $($resp., $N^{\hspace{.1mm}3}_{29})$ is obtained from a stacked
$3$-sphere by $30$ combinatorial handle additions.
\end{lemma}

\begin{proof}
We present a proof for $M^{\hspace{.2mm}3}_{29}$. Similar arguments work for $N^{\hspace{.2mm}3}_{29}$.

Consider the pure 4-dimensional simplicial complexes $B_1$ and $B_2$ on the vertex sets $V_1 = \{a_j \, : -4 \leq
j\leq 28\}$ and $V_2 = \{b_0, b_1, \dots, b_{28}, b_{34}, b_{40}, b_{46}, b_{52}\} \cup\{u_{j} \, : \, -4\leq j
\leq 24\} \cup\{v_{j} \, : \, -3\leq j \leq 25\} \cup\{w_{j} \, : \, -2\leq j\leq 26\}$, respectively, given by
\begin{align}
B_1 = \{\widetilde{\sigma}_i \, : \, 0\leq i\leq 28\}, & \, \,
B_2  = \{\widetilde{\mu}_i \, : \, 0\leq i\leq 28\} \cup
 \{\widetilde{\alpha}_{k,i} \, : \, 0\leq i\leq 28, \, 1\leq k\leq 3\},
\end{align}
where $\widetilde{\sigma}_i := \{a_{i-j} \, : 0\leq j\leq 4\}$, $\widetilde{\mu}_i := \{b_{i+5+6j} \, : 0\leq
j\leq 4\}$, $\widetilde{\alpha}_{3,i} :=\{w_{i-2}, b_i, b_{i+5}, b_{i+11}$, $b_{i+17}\}$,
$\widetilde{\alpha}_{2,i} :=\{v_{i-3}, w_{i-2}, b_i, b_{i+5}, b_{i+11}\}$, $\widetilde{\alpha}_{1,i} :=
\{u_{i-4}, v_{i-3}, w_{i-2}, b_i, b_{i+5}\}$, for $m\geq 29$ and $m\neq 34, 40, 46, 52$ we have $b_m :=
b_{m-29}$. Clearly, $B_1$ is a stacked 4-ball whose dual graph is a path. The dual graph of $B_2$ is a tree (a
comb with 29 teeth) with 116 vertices. Since $B_2$ has 120 vertices, by Lemma \ref{lemma:stackedball}, $B_2$ is a
stacked 4-ball. Let $B$ be the simplicial complex obtained from $B_1 \sqcup B_2$ by identifying the 3-simplices
$\alpha_0 = a_{-4}a_{-3}a_{-2}a_0$ and $\beta_0= u_{-4}v_{-3}w_{-2}b_0$ by the map $\psi \colon u_{-4}\mapsto
a_{-4}$, $v_{-3}\mapsto a_{-3}$, $w_{-2}\mapsto a_{-2}$, $b_0 \mapsto a_{0}$. Again by Lemma
\ref{lemma:stackedball}, $B$ is a stacked 4-ball.  Let $S := \partial B$. Then $S$ is a stacked 3-sphere with 149
vertices.

Now consider the sixty 3-simplices $\alpha_i = \{u_{i-4}, v_{i-3}, w_{i-2}, b_{i}\}$, $\beta_i = \{a_{i-4},
a_{i-3}, a_{i-2}, a_{i}\}$, $1\leq i\leq 28$, $\alpha_{29} = \{a_{-4}, a_{-3}, a_{-2}, a_{-1}\}$, $\beta_{29} =
\{a_{25}, a_{26}, a_{27}, a_{28}\}$, $\alpha_{30} = \{b_{34}, b_{40}, b_{46}, b_{52}\}$, $\beta_{30} = \{b_{5},
b_{11}, b_{17}, b_{23}\}$ and the 30 maps $\psi_{i} \colon \alpha_i \to \beta_i$ given by $\psi_i(u_{i-4}) =
a_{i-4}$, $\psi_i(v_{i-3}) = a_{i-3}$, $\psi_i(w_{i-2}) = a_{i-2}$, $\psi_i(b_{i}) = a_{i}$, for $1\leq i\leq
28$, $\psi_{29}(a_j) = a_{j+29}$, $\psi_{30}(b_j) = b_{j-29}$. For $1\leq i\leq 30$, let $M_i= M_{i-1}^{\psi_i}$,
where $M_0=S$. For $0\leq j\leq 30$, let $N_j(x)$ denote the set of neighbors of $x$ in $M_j$. Then
$N_{i-1}(a_i)\setminus \beta_i = \{a_{i+k} : -4\leq k \leq 4, k\neq 0\}$, $N_{i-1}(b_i)\setminus \alpha_i =
\{b_{i+5+6k} : 0\leq k \leq 8, k\neq 4\}$. Therefore, $N_{i-1}(a_i) \cap N_{i-1}(b_i) =\emptyset$. Similarly,
$N_{i-1}(a_{i-2}) \cap N_{i-1}(w_{i-2})  = N_{i-1}(a_{i-3}) \cap N_{i-1}(v_{i-3}) =N_{i-1}(a_{i-4}) \cap
N_{i-1}(u_{i-4}) =\emptyset$. Thus $\psi_i$ is admissible for $1\leq i\leq 28$. Similarly, we can show that
$\psi_{29}$ and $\psi_{30}$ are admissible. Since $M_0=S\in {\mathcal K}(3)$, inductively it follows that
$M_{30}\in {\mathcal K}(3)$. It is now easy to see that $M_{30}$ is isomorphic to $M^{\hspace{.1mm}3}_{29}$. This
completes the proof.
\end{proof}

If $\sim$ is the equivalence relation generated by $x\sim \psi_i(x)$ for $x\in \alpha_i$, $1\leq i\leq 30$, then
the quotient complex $B/\!\sim$ is isomorphic to ${\mathcal M}^{\hspace{.1mm}4}_{29}$, where $B$ and $\psi_i$ are
as in the above proof. In Lemma \ref{lemma:orientability}, we show that $M^{\hspace{.1mm}3}_{29}$ is
non-orientable.

\section{Construction in $\overline{\mathcal K}(d)$} \label{sec:cons}

In this section, we present constructions of neighborly members of $\overline{{\mathcal K}}(d+1)$. In particular,
we construct manifolds in $\overline{{\mathcal K}}(d+1)$ whose boundaries are $M^{\hspace{.1mm}d}_{n}$ and
$N^{\hspace{.1mm}d}_{n}$, respectively. Our constructions are based on Lemma \ref{lemma:construction} below
(\cite{bdns1}). Given a graph $G$ and a family ${\mathcal T}=\{T_i\}_{i\in {\mathcal I}}$ of induced subtrees of
$G$, we say that $u \in V(G)$ {\em defines} the subset $\hat{u} := \{i \in {\mathcal I} : u \in V(T_i)\}$ of
$\mathcal I$.

\begin{lemma}\label{lemma:construction}
Let $G$ be a finite graph and ${\mathcal T} = \{T_i\}_{i=1}^{n}$ be a family of $(n-d)$-vertex induced subtrees
of $G$. Suppose that {\rm (i)} any two of the $T_i$'s intersect, {\rm (ii)} each vertex of $G$ is in exactly
$d+1$ members of $\mathcal T$ and {\rm (iii)} for any two vertices $u\neq v$ of $G$, $u$ and $v$ are together in
exactly $d$ members of $\cal T$ if and only if $uv$ is an edge of $G$. Then the pure simplicial complex $M$ whose
facets are $\hat{u}$, where $u \in V(G)$, is an $n$-vertex neighborly member of $\overline{\mathcal K}(d)$, with
$\Lambda(M)\cong G$.
\end{lemma}

\begin{proof}
First we prove that $\hat{u} \neq \hat{v}$ for $u\neq v$ in $V(G)$. Assume that $\hat{u} = \hat{v}$ for some
$u\neq v$. Let $P$ be a $u \mbox{-}v$ path in $G$. Let $w$ be the neighbor of $u$ on $P$. Then $uw$ is an edge of
$G$ and hence $d = \#(\hat{u} \cap \hat{w}) = \#(\hat{v} \cap \hat{w})$. Therefore $wv$ is an edge in $G$. Let
$i\in \hat{u} \setminus \hat{w}= \hat{v} \setminus \hat{w}$. Let $Q$ be the $u\mbox{-}v$ path in the tree $T_i$.
Let $z$ be the neighbor of $u$ on $Q$. Since $i\not\in \hat{w}$, $z\neq w$. As before, we have $d = \#(\hat{u}
\cap \hat{z}) = \#(\hat{v} \cap \hat{z})$. Therefore, $zv$ is an edge in $G$. Since $d\geq 2$, it follows that
$\hat{u} \cap \hat{w} \cap \hat{z} \neq \emptyset$. Let $j\in \hat{u} \cap \hat{w} \cap \hat{z}$. Since $\hat{u}
= \hat{v}$, it follows that $j$ is in $\hat{u}$, $\hat{v}$, $\hat{w}$ and $\hat{z}$. Then $T_j$ contains $u, v,
w$ and $z$. Since $T_j$ is an induced subgraph it contains the cycle $uwvzu$, a contradiction to the fact that
$T_j$ is a tree.

Let ${S} \subseteq \{1, \dots, n\}$ be of size $d$. We show that at most two facets of $M$ contain $S$. If
possible, let $\hat{u}$, $\hat{v}$ and $\hat{w}$ be three facets of $M$ that contain $S$. Then by assumption,
$uv$, $uw$ and $vw$ are edges in $G$. Let $i\in {S}$. Then $u$, $v$
and $w$ are vertices of $T_i$. Since $T_i$ is
induced subgraph, we conclude that $uv$, $uw$, $vw$ are edges of $T_i$, which is a contradiction to the fact that
$T_i$ is a tree. Thus $M$ is a $d$-dimensional weak pseudomanifold. Clearly $u\mapsto \hat{u}$ is an isomorphism
between $G$ and $\Lambda(M)$. Further the conditions on $(G, \mathcal T)$ imply that $G$ is connected. Thus $M$
is a $d$-pseudomanifold. Since any two members of $\mathcal T$ intersect, it follows that $M$ is neighborly. Let
$S_i = \st{M}{i}$ be the star of the vertex $i$ in $M$. Then by construction $\Lambda(S_i) \cong T_i$ and thus
$f_d(S_i) = \#(V(T_i)) = n-d$. Also from the neighborliness of $M$, $f_0(S_i)=n$. Thus $f_0(S_i)=f_d(S_i)+d$ and
hence, by Lemma \ref{lemma:stackedball}, $S_i$ is a stacked $d$-ball. Therefore, by Corollary \ref{cor:C0},
$\lk{M}{i}$ is a stacked $(d-1)$-ball and hence $M$ is a member of $\overline{\mathcal K}(d)$.
\end{proof}

We consider two examples of intersecting families of induced subtrees
of a graph which we use to show that ${\mathcal M}^{\hspace{.1mm}d
+1}_{n}$
and ${\mathcal N}^{\hspace{.1mm}d+1}_{n}$ are in $\overline{\mathcal
K}(d+1)$ (cf. Lemmata \ref{lemma:Md} and \ref{lemma:Nd}).

\begin{eg}\label{example:Gd}
{\rm Let $d\geq 2$ and $n = d^{\hspace{.3mm}2}+5d+5$. Consider the graph $G^{\hspace{.1mm}d}$ on $n(d+2)$
vertices consisting of two $n$-cycles $C_1$, $C_2$ and $n$ disjoint paths $P_{i}$, $0\leq i\leq n-1$, given by
\begin{align}\label{eq:graph}
C_1 =\sigma_0\sigma_1\cdots\sigma_{n-1}\sigma_0, \, C_2=\mu_0\mu_{d+3}\mu_{2(d+3)}\cdots\mu_{(n-1)(d+3)}\mu_0, \,
P_i =\sigma_i\alpha_{1,i}\alpha_{2,i}\cdots\alpha_{d,i}\mu_i,
\end{align}
where the subscripts (except the first subscript on $\alpha$) are to be
taken modulo $n$. Let ${\mathcal T}_1=\{T_i\}_{i=0}^{n-1}$ be the family of induced trees where the
vertex-set $V(T_i)$ of $T_i$ is given by
\begin{align}\label{eq:inducedsets}
V(T_i) =  & \{\sigma_{i+j}: 0\leq j\leq d+1\} \cup
            \{\mu_{i+j(d+3)}: 0\leq j\leq d+1\} \cup
            \{\alpha_{j,i}: 1\leq j\leq d\}\cup \nonumber \\
 & (\bigcup_{k=2}^{d+1} \{\alpha_{j,i+k}: 1\leq j\leq d+2-k\}) \cup
 (\bigcup_{k=2}^{d+1} \{\alpha_{j,i+k(d+3)}: d+2-k\leq j\leq d\}),
\end{align}
see Figure \ref{fig:fig1}. Figure \ref{fig:fig2} shows the graph $G^{\hspace{.1mm}4}$ with the tree $T_0$ in
black.}
\end{eg}

\begin{figure}[ht]
\centering
\begin{tikzpicture}[thick]
\draw (0,0) -- (0,1); \draw [dashed] (0,1) -- (0,4); \draw (0,4) -- (0,5) -- (1,5) -- (3,5); \draw [dashed] (3,5)
-- (5,5); \draw (5,5) -- (6,5) -- (6,4)
      (0,0) -- (2.5,0);
\draw [dashed] (2.5,0) -- (5.5,0); \draw (5.5,0) -- (6.5,0) -- (6.5,1); \draw [dashed] (6.5,1) -- (6.5,3); \draw
(6.5,3) -- (6.5,4); \draw (2,5) -- (2,4); \draw [dashed] (2,4) -- (2,2); \draw (2,2) -- (2,1); \draw (3,5) --
(3,4); \draw [dashed] (3,4) -- (3,3); \draw (3,3) -- (3,2); \draw (5.5,0) -- (5.5,1); \draw [dashed] (5.5,1) --
(5.5,2); \draw (5.5,2) -- (5.5,3); \draw (2.5,0) -- (2.5,1); \draw (5,5) -- (5,3);

\draw (0,0) node {$\bullet$}; \draw (1.5,0) node {$\bullet$}; \draw (2.5,0) node {$\bullet$}; \draw (5.5,0) node
{$\bullet$}; \draw (6.5,0) node {$\bullet$}; \draw (0,1) node {$\bullet$}; \draw (0,4) node {$\bullet$}; \draw
(0,5) node {$\bullet$}; \draw (1,5) node {$\bullet$}; \draw (2,5) node {$\bullet$}; \draw (5,5) node {$\bullet$};
\draw (6,5) node {$\bullet$}; \draw (2,5) node {$\bullet$}; \draw (2,2) node {$\bullet$}; \draw (2,1) node
{$\bullet$}; \draw (3,5) node {$\bullet$}; \draw (3,4) node {$\bullet$}; \draw (3,3) node {$\bullet$}; \draw
(2,4) node {$\bullet$}; \draw (6,4) node {$\bullet$}; \draw (3,3) node {$\bullet$}; \draw (2.5,1) node
{$\bullet$}; \draw (3,2) node {$\bullet$}; \draw (6.5,1) node {$\bullet$}; \draw (6.5,3) node {$\bullet$}; \draw
(5.5,3) node {$\bullet$}; \draw (6.5,4) node {$\bullet$}; \draw (5.5,1) node {$\bullet$}; \draw (5.5,2) node
{$\bullet$}; \draw (5.5,3) node {$\bullet$}; \draw (5,4) node {$\bullet$}; \draw (5,3) node {$\bullet$};

{\small \draw (0,0) node [below left] {$\mu_i$}; \draw (1.5,0) node [below] {$\mu_{i+d+3}$}; \draw (2.5,0) node
[below right] {$\mu_{i+2(d+3)}$}; \draw (5.5,0) node [below] {$\mu_{i+d(d+3)}$}; \draw (6.5,0) node [below right]
{$\mu_{i+(d+1)(d+3)}$}; \draw (0,1) node [left] {$\alpha_{d,i}$}; \draw (0,4) node [left] {$\alpha_{1,i}$}; \draw
(0,5) node [above left] {$\sigma_i$}; \draw (1,5) node [above] {$\sigma_{i+1}$}; \draw (2,5) node [above]
{$\sigma_{i+2}$}; \draw (3,5) node [above] {$\sigma_{i+3}$}; \draw (5,5) node [above] {$\sigma_{i+d}$}; \draw
(6,5) node [above right] {$\sigma_{i+d+1}$}; \draw (2,1) node [left] {$\alpha_{d,i+2}$}; \draw (2.5,1) node
[right] {$\alpha_{d,i+2(d+3)}$}; \draw (3,2) node [right] {$\alpha_{d-1,i+3}$}; \draw (5,3) node [left]
{$\alpha_{2,i+d}$}; \draw (6.5,4) node [right] {$\alpha_{1,i+(d+1)(d+3)}$};

 }
\end{tikzpicture}
\caption{Schematic representation of the tree $T_i$ in ${\cal T}_1$.} \label{fig:fig1}
\end{figure}
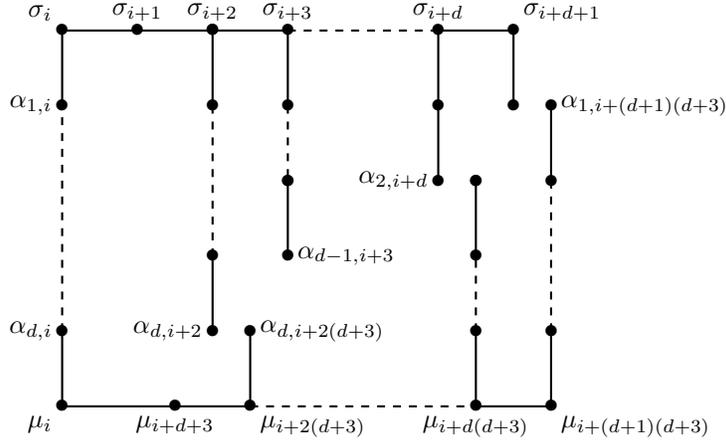


\begin{eg}\label{example:Hd}
{\rm Let $d\geq 2$ and $n = d^{\hspace{.3mm}2}+5d+5$. Let $G^{d}$ be the graph as defined in Example
\ref{example:Gd}. Further let, ${\cal T}_2 = \{T_i\}_{i=1}^n$ be the family of induced subtrees of $G^d$, where
the vertex-set $V(T_i)$, of the tree $T_i$ is given by (see Figure \ref{fig:fig3}):
\begin{align}\label{eq:inducedsets2}
V(T_i) =  & \{\sigma_{i+j}: 0\leq j\leq d+1\} \cup
            \{\mu_{i+j(d+3)}: 0\leq j\leq d+1\} \cup
            \{\alpha_{j,i}: 1\leq j\leq d\}\cup \nonumber \\
          & (\bigcup_{k=2}^{d+1} \{\alpha_{j,i+k}: 1\leq j\leq d+2-k\})
           \cup(\bigcup_{k=2}^{d+1} \{\alpha_{j,i+k(d+3)} :
k-1\leq j\leq d\}).
\end{align}
}
\end{eg}

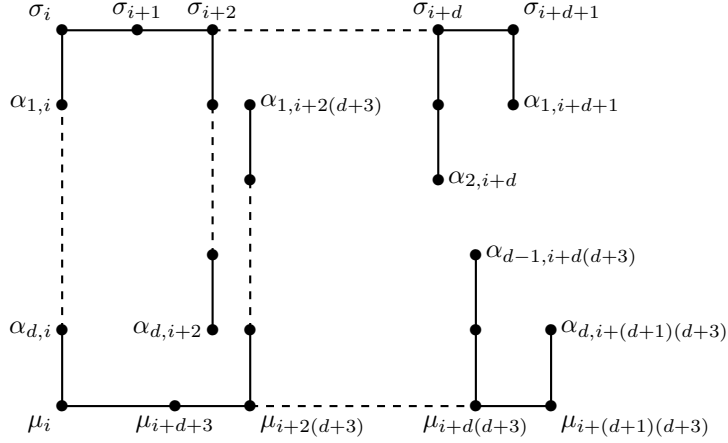
\begin{figure}[hb]
\centering
\begin{tikzpicture}[thick]
\draw (0,0) -- (0,1); \draw [dashed] (0,1) -- (0,4); \draw (0,4) -- (0,5) -- (1,5) -- (2,5); \draw [dashed] (2,5)
-- (5,5); \draw (5,5) -- (6,5) -- (6,4) (0,0) -- (2.5,0); \draw [dashed] (2.5,0) -- (5.5,0); \draw (5.5,0) --
(6.5,0) -- (6.5,1);

\draw (2,5) -- (2,4); \draw [dashed] (2,4) -- (2,2); \draw (2,2) -- (2,1); \draw (2.5,0) -- (2.5,1); \draw
[dashed] (2.5,1) -- (2.5,3); \draw (2.5,3) -- (2.5,4); \draw (5.5,0) -- (5.5,2); \draw (5,5) -- (5,3);

\draw (0,0) node {$\bullet$}; \draw (1.5,0) node {$\bullet$}; \draw (2.5,0) node {$\bullet$}; \draw (5.5,0) node
{$\bullet$}; \draw (6.5,0) node {$\bullet$}; \draw (0,1) node {$\bullet$}; \draw (0,4) node {$\bullet$}; \draw
(0,5) node {$\bullet$}; \draw (1,5) node {$\bullet$}; \draw (2,5) node {$\bullet$}; \draw (5,5) node {$\bullet$};
\draw (6,5) node {$\bullet$}; \draw (2,2) node {$\bullet$}; \draw (2,1) node {$\bullet$}; \draw (2,4) node
{$\bullet$}; \draw (6,4) node {$\bullet$}; \draw (2.5,1) node {$\bullet$}; \draw (2.5,3) node {$\bullet$}; \draw
(2.5,4) node {$\bullet$}; \draw (6.5,1) node {$\bullet$}; \draw (5.5,1) node {$\bullet$}; \draw (5.5,2) node
{$\bullet$}; \draw (5,4) node {$\bullet$}; \draw (5,3) node {$\bullet$};

{\small \draw (0,0) node [below left] {$\mu_i$}; \draw (1.5,0) node [below] {$\mu_{i+d+3}$}; \draw (2.5,0) node
[below right] {$\mu_{i+2(d+3)}$}; \draw (5.5,0) node [below] {$\mu_{i+d(d+3)}$}; \draw (6.5,0) node [below right]
{$\mu_{i+(d+1)(d+3)}$}; \draw (0,1) node [left] {$\alpha_{d,i}$}; \draw (0,4) node [left] {$\alpha_{1,i}$}; \draw
(0,5) node [above left] {$\sigma_i$}; \draw (1,5) node [above] {$\sigma_{i+1}$}; \draw (2,5) node [above]
{$\sigma_{i+2}$}; \draw (5,5) node [above] {$\sigma_{i+d}$}; \draw (6,5) node [above right] {$\sigma_{i+d+1}$};
\draw (2,1) node [left] {$\alpha_{d,i+2}$}; \draw (6.5,1) node [right] {$\alpha_{d,i+(d+1)(d+3)}$}; \draw (6,4)
node [right] {$\alpha_{1,i+d+1}$}; \draw (5,3) node [right] {$\alpha_{2,i+d}$}; \draw (5.5,2) node [right]
{$\alpha_{d-1,i+d(d+3)}$}; \draw (2.5,4) node [right] {$\alpha_{1,i+2(d+3)}$};
 }
\end{tikzpicture}
\caption{Schematic representation of the tree $T_i$ in
 ${\cal T}_2$.}
\label{fig:fig3}
\end{figure}


\begin{lemma}\label{lemma:neighborly}
Let $d\geq 2$ and $n = d^{\hspace{.3mm}2}+5d+5$. Let the graph $G^{\hspace{.1mm}d}$ and
the family of induced subtrees ${\mathcal
T}_1=\{T_i\}_{i=0}^{n-1}$ be as in Example $\ref{example:Gd}$. Then
$T_i\cap T_j\neq \emptyset$ for all $0\leq i,j\leq n-1$.
\end{lemma}

\begin{proof} Let $\varphi$ be a bijection on $V(G^{\hspace{.1mm}d})$ given by \vspace{-4mm}
\begin{equation}
\varphi = (\sigma_0,\ldots,\sigma_{n-1})(\mu_0,\ldots,\mu_{n-1})
\prod_{j=1}^{d}(\alpha_{j,0},\ldots,\alpha_{j,n-1}). \nonumber
\end{equation}
It is easily seen that $\varphi$ is an automorphism of $G^{\hspace{.1mm}d}$ and further we have
$T_{i+1}=\varphi(T_i)$. Thus $T_i= \varphi^i(T_0)$. Thus to show that $T_i\cap T_j\neq \emptyset$ for $0\leq
i,j\leq n-1$, it is sufficient to show that $T_i\cap T_0\neq \emptyset$ for $0\leq i\leq n-1$.

\smallskip

\par\noindent{\bf Claim\,:} For $0\leq i\leq n-1$, if there exist
integers $l, k$ with $2\leq l\leq k\leq d+1$ which satisfy either (i) $i+k(d+3)=n+l$ or (ii) $i+l=k(d+3)$ then
$T_i$ intersects $T_0$.

\smallskip

Suppose $i+k(d+3)=n+l$ for some integers $l,k$ satisfying $2\leq l\leq k\leq d+1$. Thus $i+k(d+3)\equiv l$ (mod
$n$). Then from (\ref{eq:inducedsets}), we see that $\{\alpha_{j,l}:d+2-k\leq j\leq d\}\subseteq V(T_i)$. Also
from (\ref{eq:inducedsets}), $\{\alpha_{j,l}:1\leq j\leq d+2-l\}\subseteq V(T_0)$. For $l\leq k$, we see that the
intersection of the above two sets is $\{\alpha_{j,l} : d+2-k\leq j\leq d+2-l\} \subseteq V(P_l)$. Next suppose
that $i+l=k(d+3)$ for some integers $2\leq l\leq k\leq d+1$. Again from (\ref{eq:inducedsets}), we have
$\{\alpha_{j,i+l}: 1\leq j\leq d+2-l\} \subseteq V(T_i)$ and $\{\alpha_{j,i+l}: d+2-k\leq j\leq d\}\subseteq
V(T_0)$. Thus for $l\leq k$, the two sets intersect, and hence $T_0$ and $T_i$ intersect. This proves the claim.

\smallskip

Clearly, we have the following six cases. \begin{enumerate} \item[(a)] $0\leq i \leq d+1$: In this case, $T_i$
intersects $T_0$ in $\sigma_i$.

\item[(b)] $i> (d+1)(d+3)$: It is easy to see that $T_i$ contains $\sigma_0$ and hence $T_i \cap T_0 \neq
\emptyset$.

\item[(c)] $i= k(d+3)$, $1\leq k\leq d+1$: In this case, $T_i$ intersects $T_0$ in $\mu_i$.

\item[(d)] $i=k(d+3)-1$, $1\leq k\leq d+1$: Then $i+l(d+3)\equiv n\equiv 0$ (mod $n$), where $l=d+2-k \leq d+1$.
This implies $T_i$ contains $\mu_0\in V(T_0)$. So, $T_i \cap T_0 \neq \emptyset$.

\item[(e)] $j(d+3)<i<(j+1)(d+2)$, $1\leq j \leq d$: Let $i=j(d+3)+t$, where $1\leq t<d+2-j$. Let $k=d+2-j$. Then
$i + k(d +3)=n+l$ where $l=t+1$. Since $1\leq j\leq d$, we have $2 \leq l \leq k\leq d+1$. Hence, by the claim,
$T_i$ intersects $T_0$.

\item[(f)] $k(d+2)\leq i<k(d+3)-1$, $2 \leq k\leq d+1$: Let $i=k(d+2)+t$ where $0\leq t<k-1$. Let $l=k-t$. Then
$i+l= k(d+3)$ and $2\leq l\leq k\leq d+1$. Therefore, by the claim, $T_i$ intersects $T_0$.

\end{enumerate}
This completes the proof of the lemma.
\end{proof}

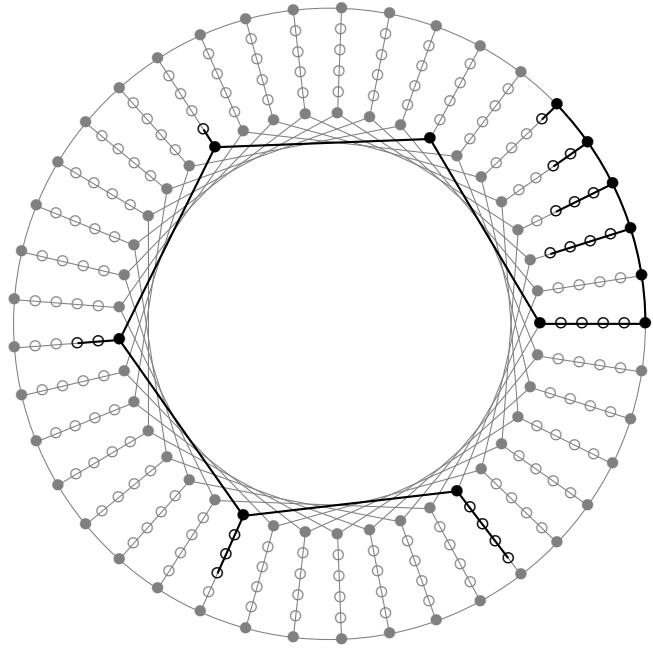
\begin{figure}[tp]
\centering
\begin{tikzpicture}[gray,thin,scale=1.4]
\draw (0,0) circle (3.0cm); \foreach \i in {0,...,40} { \draw ({(360/41*\i)}:3.0cm) node {$\bullet$}; \draw
({(360/41*\i)}:2.0cm) node {$\bullet$}; \draw ({(360/41*\i)}:2.0cm) -- ({(360/41*\i)}:3.0cm); \draw
({(360*7/41*\i)}:2.0cm) -- ({(360*7/41 *(\i+1))}:2.0cm); \draw ({(360/41*\i)}:2.2cm) node {$\circ$}; \draw
({(360/41*\i)}:2.4cm) node {$\circ$}; \draw ({(360/41*\i)}:2.6cm) node {$\circ$}; \draw ({(360/41*\i)}:2.8cm)
node {$\circ$}; }

    \foreach \i in {0} {
      \draw ({(360/41*\i)}:3.0cm) node {\color{black}{$\bullet$}}
            ({(360/41*\i)}:2.8cm) node {\color{black}{$\circ$}}
             ({(360/41*\i)}:2.6cm) node {\color{black}{$\circ$}}
             ({(360/41*\i)}:2.4cm) node {\color{black}{$\circ$}}
             ({(360/41*\i)}:2.2cm) node {\color{black}{$\circ$}}
            ({(360/41*\i)}:2.0cm) node {\color{black}{$\bullet$}}
            ({(360/41*(\i+1))}:3.0cm) node {\color{black}{$\bullet$}}
            ({(360/41*(\i+2))}:3.0cm) node {\color{black}{$\bullet$}}
            ({(360/41*(\i+2))}:2.8cm) node {\color{black}{$\circ$}}
            ({(360/41*(\i+2))}:2.6cm) node {\color{black}{$\circ$}}
            ({(360/41*(\i+2))}:2.4cm) node {\color{black}{$\circ$}}
            ({(360/41*(\i+2))}:2.2cm) node {\color{black}{$\circ$}}
            ({(360/41*(\i+3))}:3.0cm) node {\color{black}{$\bullet$}}
            ({(360/41*(\i+3))}:2.8cm) node {\color{black}{$\circ$}}
            ({(360/41*(\i+3))}:2.6cm) node {\color{black}{$\circ$}}
            ({(360/41*(\i+3))}:2.4cm) node {\color{black}{$\circ$}}
            ({(360/41*(\i+4))}:3.0cm) node {\color{black}{$\bullet$}}
            ({(360/41*(\i+4))}:2.8cm) node {\color{black}{$\circ$}}
            ({(360/41*(\i+4))}:2.6cm) node {\color{black}{$\circ$}}
            ({(360/41*(\i+5))}:3.0cm) node {\color{black}{$\bullet$}}
            ({(360/41*(\i+5))}:2.8cm) node {\color{black}{$\circ$}};

    \draw ({(360/41*(\i+7))}:2.0cm) node {\color{black}{$\bullet$}}
          ({(360/41*(\i+14))}:2.0cm) node {\color{black}{$\bullet$}}
          ({(360/41*(\i+14))}:2.2cm) node {\color{black}{$\circ$}}
          ({(360/41*(\i+21))}:2.0cm) node {\color{black}{$\bullet$}}
          ({(360/41*(\i+21))}:2.2cm) node {\color{black}{$\circ$}}
          ({(360/41*(\i+21))}:2.4cm) node {\color{black}{$\circ$}}
          ({(360/41*(\i+28))}:2.0cm) node {\color{black}{$\bullet$}}
          ({(360/41*(\i+28))}:2.2cm) node {\color{black}{$\circ$}}
          ({(360/41*(\i+28))}:2.4cm) node {\color{black}{$\circ$}}
          ({(360/41*(\i+28))}:2.6cm) node {\color{black}{$\circ$}}
          ({(360/41*(\i+35))}:2.0cm) node {\color{black}{$\bullet$}}
          ({(360/41*(\i+35))}:2.2cm) node {\color{black}{$\circ$}}
          ({(360/41*(\i+35))}:2.4cm) node {\color{black}{$\circ$}}
          ({(360/41*(\i+35))}:2.6cm) node {\color{black}{$\circ$}}
          ({(360/41*(\i+35))}:2.8cm) node {\color{black}{$\circ$}};

    }

    \foreach \i in {0} {
        \draw [black,thick] ({(360/41*\i)}:3.0cm) arc ({360/41*\i}:{360/41*(\i+5)}:3.0cm)
        ({(360/41*\i)}:3.0cm) -- ({(360/41*\i)}:2.0cm)
        -- ({360/41*(\i+7)}:2.0cm)
        -- ({360/41*(\i+14)}:2.0cm)
        -- ({360/41*(\i+21)}:2.0cm)
        -- ({360/41*(\i+28)}:2.0cm)
        -- ({360/41*(\i+35)}:2.0cm)
        ({360/41*(\i+2)}:3.0cm) -- ({360/41*(\i+2)}:2.2cm)
        ({360/41*(\i+3)}:3.0cm) -- ({360/41*(\i+3)}:2.4cm)
        ({360/41*(\i+4)}:3.0cm) -- ({360/41*(\i+4)}:2.6cm)
        ({360/41*(\i+5)}:3.0cm) -- ({360/41*(\i+5)}:2.8cm)
        ({360/41*(\i+14)}:2.0cm) -- ({360/41*(\i+14)}:2.2cm)
        ({360/41*(\i+21)}:2.0cm) -- ({360/41*(\i+21)}:2.4cm)
        ({360/41*(\i+28)}:2.0cm) -- ({360/41*(\i+28)}:2.6cm)
        ({360/41*(\i+35)}:2.0cm) -- ({360/41*(\i+35)}:2.8cm);
    }
\end{tikzpicture}
\caption{Graph $G^4$ and the tree $T_0$ ($\in {\mathcal T}_1$) in black} \label{fig:fig2}
\end{figure}

\begin{lemma}\label{lemma:pure}
Let $d\geq 2$ and $n = d^{\hspace{.3mm}2}+5d+5$. Let the graph $G^{\hspace{.1mm}d}$ and
the family of induced subtrees ${\mathcal
T}_1=\{T_i\}_{i=0}^{n-1}$ be as in Example $\ref{example:Gd}$. Then
\begin{enumerate}[{\rm (a)}]
\item $T_i$ is a tree on $n-d-1$ vertices.
\item For all $v\in V(G^{\hspace{.1mm}d})$, $\hat{v}$ is a $(d+
2)$-element set.
\item For $u,v\in V(G^{\hspace{.1mm}d})$, $\hat{u}\cap \hat{v}$ is a $(d+1)$-element set if and only if $uv$ is
an edge in $G^{\hspace{.1mm}d}$.

\end{enumerate}
\end{lemma}


\begin{proof}
From (\ref{eq:inducedsets}) we have \vspace{-4mm}
$$
\#V(T_i)=(d+2)+(d+2)+d+\sum_{k=2}^{d+1}(d+2-k) + \sum_{k=2}^{d+1} (k-1)
 = d^2+4d+4=n-d-1.
$$
This proves (a). Recall that for $v\in V(G^d)$, $\hat{v} = \{i \, : \, v\in T_i\}$. Then from
(\ref{eq:inducedsets}), we see that
\begin{align} \label{eq:facets1}
\hat{\alpha}_{l,m}  = & \{m\}\cup \{m-k : 2\leq k\leq d+2-l\}\cup \{m-(d+3)j: d+2-l\leq j\leq
d+1\}, \nonumber \\
\hat{\sigma}_m  = & \{m-k: 0\leq k\leq d+1\}, \, \, \hat{\mu}_m =\{m-k(d+3): 0\leq k\leq d+1\},
\end{align}
for $1\leq l\leq d$ and $0\leq m\leq n-1$. (Here $\hat{\sigma}_m, \hat{\mu}_m, \hat{\alpha}_{l,m} \subseteq
\mathbb{Z}_n$.) Clearly $\hat{\sigma}_l,\hat{\mu}_l$ are sets of size $d+2$. Further, we have: for $2\leq k,
j\leq d+1$, $k\not\equiv (d+3)j$ (mod $n$), and hence $\#(\hat{\alpha}_{l, m})= 1+ (d+1-l) +l =d+2$. This proves
(b).

Let us define a metric $\Delta$ on the set $V(G^{\hspace{.1mm}d})$ as $\Delta(u,v) := \#(\hat{u}\backslash
\hat{v})= \#(\hat{v} \backslash \hat{u})$. It is easy to see that $\Delta$ indeed defines a metric on
$V(G^{\hspace{.1mm}d})$, the proof of which will be omitted here. Clearly, $\#(\hat{\sigma}_i \cap \hat{\mu}_j) <
d+1$ and $\sigma_i\mu_j$ is not an edge of $G^{\hspace{.1mm}d}$ for $0\leq i, j\leq n-1$. Thus, to prove (c), we
need to show the following:

\begin{enumerate}[{\rm (i)}]
\item $\Delta(\sigma_i,\sigma_j)=1 \Leftrightarrow i-j\equiv \pm 1$ (mod $n$). \item $\Delta(\mu_i,\mu_j)=1
\Leftrightarrow i-j\equiv\pm (d+3)$ (mod $n$). \item $\Delta(\alpha_{l,m},\alpha_{r,s})=1$ $\Leftrightarrow$
$m=s$, $l-r=\pm 1$. \item $\Delta(\sigma_i,\alpha_{l,m})=1 \Leftrightarrow i=m, l=1$. \item
$\Delta(\mu_i,\alpha_{l,m})=1 \Leftrightarrow
 i=m, l=d$.
\end{enumerate}

In all the above cases, the reverse implications follow from the definitions of the sets in (\ref{eq:facets1}).
Before we proceed with the proofs of the forward implications, we introduce some notation. For integers $i,j$,
let $|i-j|_n$ denote the smallest non-negative integer $k$ such that either $i+k\equiv j$ (mod $n$), or
$j+k\equiv i$ (mod $n$). If we think of $\mathbb{Z}_n$ as the vertex set of the $n$-cycle $C_n$ whose edges are
$\{i, i+1\}$ then $|i-j|_n$ is the distance between vertices $i$ and $j$ in $C_n$. Thus, $| \cdot |_n$ is a
metric on $\mathbb{Z}_n$ and $|i-j|_n\leq n/2$ for all $i, j$. For integers $i\leq j$, let $[i,j]_n :=\{z\in
\mathbb{Z} : z\equiv k$ (mod $n$), for some $k \in \{i, i+1, \dots, j\}\}$.

\smallskip

\par\noindent {\bf Claim 1:} $\Delta(\sigma_i,\sigma_j)\geq
\min\{|i-j|_n, \, d+2\}$.

\smallskip

If $|i-j|_n = 0$ then there is nothing to prove. So, assume that $t :=|i-j|_n>0$.

Assume, without loss, that $j\equiv i+t$ (mod $n$). Let $T=\{j-k: 0\leq k\leq t-1\}$. We claim that $T\cap
\hat{\sigma}_i = \emptyset$. Assume that $T\cap \hat{\sigma}_i \neq \emptyset$. Then there exist integers
$k,k'$, where $0\leq k\leq t-1$ and $0\leq k'\leq d+1$, such that $j-k\equiv i- k'$ (mod $n$). So, $j\equiv
i+(k-k')$ (mod $n$). Since $k-k'\leq t-1$, this implies (by the definition of $|i-j|_n$) that $k-k'<0$. Thus,
$t-(k-k') >0$ and $t-(k-k') \equiv 0$ (mod $n$) (since $t = j-i \equiv k-k'$ (mod $n$)). So, $t-(k-k') = pn$ for
some positive integer $p$. Then $n \leq pn \leq t+k' \leq n/2+(d+1)$. This implies $n \leq 2d+2$, a
contradiction. Thus, $T\cap \hat{\sigma}_i= \emptyset$.

Now, if $t \leq d+1$, then $T \subseteq \{j-k : 0\leq k\leq d+1\} = \hat{\sigma}_j$ and hence $\Delta(\sigma_i,
\sigma_j) \geq \#(T)=t$. On the other hand, if $t\geq d+2$, then $T \supseteq \{j-k : 0\leq k\leq d+1\} =
\hat{\sigma}_j$, and hence $\hat{\sigma}_i\cap \hat{\sigma}_j=\emptyset$. Therefore $\Delta(\sigma_i, \sigma_j) =
d+2$.  This proves Claim 1.

\smallskip

For $1\leq i\leq d$, $0\leq j\leq n-1$, let $A_{i,j} :=\{j\}\cup \{j-k: 2\leq k\leq d+2-i\}$, $B_{i,j}
:=\{j-k(d+3): d+2-i \leq k\leq d+1\}$, $C_{i,j} :=\{j\}\cup \{j-k(d+3): d+2-i\leq k\leq d+1\}$ and $D_{i,j}
:=\{j-k: 2\leq k\leq d+2-i\}$. So, $\hat{\alpha}_{i, j} =A_{i,j}\sqcup B_{i,j} = C_{i,j}\sqcup D_{i,j}$.

\smallskip

\par\noindent {\bf Claim 2:} (a) If $|m-s|_n>d+1$ then $A_{l,m}\cap A_{r,s}=\emptyset$ and $\#(A_{l,m}\cap
B_{r,s})\leq 1$ for $1\leq l, r\leq d$. (b) If $0< |m-s|_n\leq d+1$ then $C_{l,m}\cap C_{r,s}=\emptyset$ and
$\#(C_{l,m}\cap D_{r,s})\leq 1$  for $1\leq l, r\leq d$.

\smallskip

Suppose $|m-s|_n>d+1$. Assume that $z\in A_{l,m}\cap A_{r,s}$. Then there exist integers $k, k'$ with $0\leq
k\leq d+2-l\leq d+1$ and $0\leq k'\leq d+2-r\leq d+1$ such that $m-k \equiv z \equiv s-k'$ (mod $n$). Then
$|m-s|_n\leq d+1$, a contradiction. Thus $A_{l,m}\cap A_{r,s}=\emptyset$.

Assume that $z, x\in A_{l, m} \cap B_{r, s}$, where $z \neq x$. Since $z, x\in A_{l,m}$, there exist $a, b
\in \{0, \dots, d+1\}$ such that $z=m-a$ and $x=m-b$. Then $z-x = b-a\in [-(d+1), (d+1)]_n$. Since $z, x\in
B_{r,s}$, there exist $k, k'\in \{2, \dots, d+1\}$ such that $z\equiv s-k(d+3)$ (mod $n$) and $x\equiv s-k'(d+3)$
(mod $n$). So, $z-x\equiv (k'-k)(d+3)$ (mod $n)$. Assume without loss that $k'>k$. Then $1\leq k'-k\leq d-1$ and
hence $d+1 < (k'-k)(d+3) < n-(d+1)$. This implies that $z-x=(k'-k)(d+3)\not\in [-(d+1),(d+1)]_n$, a
contradiction. Therefore, $\#(A_{l,m}\cap B_{r,s})\leq 1$. This proves part (a). By similar arguments, part (b)
of Claim 2 follows.

\smallskip

\par\noindent {\bf Claim 3:} If $\Delta(\alpha_{l,m},\alpha_{r,s}) =1$ then $m= s$.

\smallskip

Assume that $\Delta(\alpha_{l,m},\alpha_{r,s}) =1$. Then $\#(\hat{\alpha}_{l,m}\cap \hat{\alpha}_{r,s})=d+1$.
Assume that $m \neq s$. Then $|m-s|_n> 0$. We have the following two cases.

\smallskip

\noindent {\bf Case 1.} $|m-s|_n>d+1$. Then, by Claim 2 (a), we have $A_{l,m} \cap A_{r,s} = \emptyset$,
$\#(A_{l,m} \cap \hat{\alpha}_{r, s})\leq 1$ and $\#(A_{r,s} \cap \hat{\alpha}_{l, m}) \leq 1$. Also,
$\#(B_{l,m})= l \leq d$, $\#(B_{r,s}) = r \leq d$ and hence $\#(B_{l,m} \cap B_{r,s}) \leq d$. Since
$\#(\hat{\alpha}_{l,m}\cap \hat{\alpha}_{r,s})=d+1$, these imply  $\#(B_{l,m}\cap B_{r,s})=d$. This implies
$B_{l,m}=B_{r,s}$ and $\#(B_{l,m})= d = \#(B_{r,s})$. Therefore $l=d=r$. In particular, $B_{d, m} = B_{d, s}$.
Then there exist integers $2\leq k, k'\leq d+1$ such that $m-2(d+3) \equiv s-k(d+3)$ (mod $n$), and $m-k'(d+3)
\equiv s-2(d+3)$ (mod $n$). Subtracting we get $(k'-2)(d+3)\equiv (2-k)(d+3)$ (mod $n$). Multiplying by $d+2$, we
get $k'-2\equiv 2-k$ (mod $n$) and hence $k+k'\equiv 4$ (mod $n$). Since $4 \leq k+k' \leq 2d+2 < n$, it follows
that $k=k'=2$. Thus $m-2(d+3) \equiv s-2(d+3)$ (mod $n$) and hence $m\equiv s$ (mod $n$). This is not possible
since $0\leq m, s\leq n-1$ and $m \neq s$.

\smallskip
\noindent {\bf Case 2.} $0< |m-s|_n \leq d+1$. Then, by Claim 2 (b), we have $C_{l,m}\cap C_{r,s} = \emptyset$,
$\#(C_{l,m}\cap \hat{\alpha}_{r,s}) \leq 1$ and $\#(C_{r,s} \cap \hat{\alpha}_{l,m})\leq 1$. Since
$\#(\hat{\alpha}_{l,m}\cap \hat{\alpha}_{r,s})=d+1$, we must have $\#(D_{l,m} \cap D_{r,s})=d$. This implies (as
in Case 1) $D_{l,m}=D_{r,s}$ and $\#(D_{l,m}) = d = \#(D_{r,s})$. Then, from the definition of $D_{l,m}$ (resp.,
$D_{r,s}$), $l=r=1$. So, $D_{1,m} = D_{1,s}$. As in Case 1, we get $m\equiv s$ (mod $n$). Again this is not
possible.

Thus, we get contradictions in both cases. Therefore, $m=s$. This proves Claim 3.

\smallskip

If $\Delta(\sigma_i,\sigma_j)=1$ then, by Claim 1, $|i-j|_n \leq 1$ and hence $i-j\equiv\pm 1$ (mod $n$). This
proves (i).

Since $(d+2)(d+3)\equiv 1$ (mod $n$), we see that the map $\pi:\mathbb{Z}_n\rightarrow \mathbb{Z}_n$ given by
$i\mapsto (d+2)i$ is a bijection, with the inverse map $\pi^{-1}$ given by $i\mapsto (d+3)i$. From the
definitions of $\hat{\sigma}_i$ and $\hat{\mu}_i$, we see that $\pi(\hat{\mu}_i)=\hat{\sigma}_{i(d +2)}$. Thus
$\Delta(\mu_i,\mu_j)=1 \Leftrightarrow \Delta(\sigma_{i(d +2)},\sigma_{j(d+2)})=1\Leftrightarrow |i(d+2)- j(d
+2)|_n=1$. Thus $i(d+2)-j(d+2)\equiv\pm 1$ (mod $n$), where multiplying by $d+3$ gives $i-j\equiv\pm (d+3)$ (mod
$n$). This proves (ii).

Now, assume $\Delta(\alpha_{l, m}, \alpha_{r,s})=1$. By Claim 3, $m=s$. So, $\Delta(\alpha_{l, m},\alpha_{r,m})
=1$. From the reverse implication, we have $\Delta(u, v) =1$ whenever $uv$ is an edge in $G^{\hspace{.1mm}d}$.
Notice that $\hat{\sigma}_m\cap \hat{\mu}_m=\{m\}$. Thus $\Delta(\sigma_m,\mu_m) =d+1$. Assume that $l\neq r\pm
1$. Then we can assume, without loss, that $r>l$ and $r-l\geq 2$. Then by the triangle inequality we have $d+1=
\Delta(\sigma_m, \mu_m) \leq \Delta(\sigma_m, \alpha_{l,m})+ \Delta(\alpha_{l,m}, \alpha_{r,m}) +
\Delta(\alpha_{r, m}, \mu_m) \leq l+1+(d+1-r)<d+1$, a contradiction. Therefore,  $l=r\pm 1$. This completes the
proof of (iii).

By again using the triangle inequality for $\Delta$ along the path
$P_m=\sigma_m\alpha_{1,m}\cdots\alpha_{d,m}\mu_m$, we can prove (iv)
and (v). This completes the proof of the lemma.
\end{proof}

\begin{lemma}\label{lemma:Md}
For $d\geq 2$ and $n = d^{\hspace{.1mm}2}+5d+5$, the simplicial complex ${\mathcal M}^{\hspace{.1mm}d+1}_{n}$
defined in Example $\ref{example:Md}$ is a neighborly member of $\overline{\mathcal K}(d+1)$.
\end{lemma}

\begin{proof}
Let $(G^{\hspace{.1mm}d}, {\mathcal T}_1)$ be as in Example \ref{example:Gd}. By Lemma \ref{lemma:pure},
$\hat{v}= \{i\, : \, v\in T_i\}$ is a set of $d+2$ elements for each $v\in V(G^{\hspace{.1mm}d})$. Consider the
$(d+1)$-dimensional simplicial complex $M(G^{\hspace{.1mm}d})$ consisting of facets $\hat{v}$, $v\in
V(G^{\hspace{.1mm}d})$. From Lemmata \ref{lemma:neighborly} and \ref{lemma:pure}, we see that
$(G^{\hspace{.1mm}d}, {\mathcal T}_1)$ satisfy the hypothesis of Lemma \ref{lemma:construction} and hence, by
Lemma \ref{lemma:construction}, $M(G^{\hspace{.1mm}d})$ is a neighborly member of $\overline{\mathcal K}(d+1)$.
Since $\{i+j(d+3)-1 \, : \, 1\leq j\leq k\} = \{i-j(d+3) \, : \, d+2-k \leq j \leq d+1\}$ as subsets of
$\mathbb{Z}_n$ for $0\leq i\leq n-1$ and $1\leq k\leq d+1$, by (\ref{eq:facets1}), $i \mapsto a_i$ defines an
isomorphism from $M(G^{\hspace{.1mm}d})$ to ${\mathcal M}^{\hspace{.1mm}d+ 1}_{n}$. This proves the lemma.
\end{proof}

\begin{lemma}\label{lemma:Nd}
For $d\geq 2$ and $n = d^{\hspace{.3mm}2}+5d+5$, the simplicial complex ${\mathcal N}^{\hspace{.1mm}d+1}_{n}$
defined in Example $\ref{example:Nd}$ is a neighborly member of $\overline{\mathcal K}(d+1)$.
\end{lemma}

\begin{proof} Let $(G^{\hspace{.1mm}d}, {\mathcal T}_2)$ be as in Example \ref{example:Hd}.
It can be shown (as in Lemmata \ref{lemma:neighborly} and \ref{lemma:pure}), that $(G^{\hspace{.1mm}d},{\mathcal
T}_2)$ satisfies the hypothesis of Lemma \ref{lemma:construction}. Therefore, the simplicial complex
$N(G^{\hspace{.1mm}d})$ consisting of facets $\hat{v}$, $v\in V(G^{\hspace{.1mm}d})$, is a neighborly member of
$\overline{\mathcal K}(d+1)$. Again, $i \mapsto a_i$ defines an isomorphism from $N(G^{\hspace{.1mm}d})$ to
${\mathcal N}^{\hspace{.1mm}d+ 1}_{n}$. This proves the lemma.
\end{proof}

We know that $\mathbb{Z}_n$ acts vertex-transitively on each of ${\mathcal M}^{\hspace{.1mm}d+1}_{n}$, ${\mathcal
N}^{\hspace{.1mm}d+1}_{n}$, $M^{\hspace{.1mm}d}_{n}$ and $N^{\hspace{.1mm}d}_{n}$, respectively (see Lemma
\ref{lemma:aut(MdNd1)}). Here we prove

\begin{lemma}\label{lemma:aut(MdNd2)}
Let $\psi$ be the map given in $(\ref{eq:psi})$. Then ${\rm Aut}(M^{\hspace{.1mm}d}_{n}) = {\rm Aut}({\mathcal
M}^{\hspace{.1mm}d+1}_{n}) = \langle\psi\rangle = {\rm Aut}({\mathcal N}^{\hspace{.1mm}d+1}_{n}) = {\rm
Aut}(N^{\hspace{.1mm}d}_{n}) \cong\mathbb{Z}_n$.
\end{lemma}

\begin{proof}
Since $\psi$ is an automorphism of ${\mathcal M}^{\hspace{.1mm}d+1}_{n}$ (resp., ${\mathcal
N}^{\hspace{.1mm}d+1}_{n}$), it follows that ${\rm Aut}(M^{\hspace{.1mm}d}_{n}) \supseteq {\rm Aut}({\mathcal
M}^{\hspace{.1mm}d+1}_{n}) \supseteq \langle\psi\rangle \subseteq {\rm Aut}({\mathcal N}^{\hspace{.1mm}d+1}_{n})
\subseteq {\rm Aut}(N^{\hspace{.1mm}d}_{n})$.

Let $\beta \in {\rm Aut}({\mathcal M}^{\hspace{.1mm}d+1}_{n})$ and let $\bar{\beta}\in {\rm
Aut}(\Lambda({\mathcal M}^{\hspace{.1mm}d+1}_{n}))$ be the induced automorphism. If $\beta(a_0) = a_0$ then
$\beta({\rm lk}_{{\mathcal M}^{\hspace{.1mm}d+1}_{n}}(a_0)) = {\rm lk}_{{\mathcal M}^{\hspace{.1mm}d
+1}_{n}}(a_0)$ and hence $\bar{\beta}(T_0) = T_0$ and $\bar{\beta}$ is an automorphism of the tree $T_0$. Then
$\bar{\beta}(\sigma_i) = \sigma_i$, $\bar{\beta}(\mu_{i(d+3)}) = \mu_{i(d+3)}$ and $\bar{\beta}(\alpha_{j, 0}) =
\alpha_{j,0}$ for $0\leq i\leq d+1$, $1\leq j\leq d$. These imply $\bar{\beta}|_{C_1} = {\rm Id}$,
$\bar{\beta}|_{C_2} = {\rm Id}$ and this in term implies that $\bar{\beta}$ is the identity of ${\rm
Aut}(\Lambda({\mathcal M}^{\hspace{.1mm}d+1}_{n}))$. Then, by Lemma \ref{lemma:dualgraph}, $\beta$ is the
identity of ${\rm Aut}({\mathcal M}^{\hspace{.1mm}d+1}_{n})$. Thus the only automorphism of ${\mathcal
M}^{\hspace{.1mm}d+1}_{n}$ which fixes $a_0$ is the identity. Since $\langle\alpha\rangle$ is transitive on
$V({\mathcal M}^{\hspace{.1mm}d+1}_{n})$, this implies that $\langle\alpha\rangle = {\rm Aut}({\mathcal
M}^{\hspace{.1mm}d+1}_{n})$.

Let $\beta \in {\rm Aut}({\mathcal N}^{\hspace{.1mm}d+1}_{n})$. Let $\bar{\beta}\in {\rm Aut}(\Lambda({\mathcal
N}^{\hspace{.1mm}d+1}_{n}))$ be the induced automorphism. Then $\bar{\beta}(T_i) = T_j$, where $\beta(a_i) =
a_j$, for all $i$. If $\beta(a_0) = a_0$ then $\beta({\rm lk}_{{\mathcal N}^{\hspace{.1mm}d+1}_{n}}(a_0)) = {\rm
lk}_{{\mathcal N}^{\hspace{.1mm}d +1}_{n}}(a_0)$ and hence $\bar{\beta}(T_0) = T_0$ and $\bar{\beta}$ is an
automorphism of the tree $T_0$. This implies that $\bar{\beta}(\{\sigma_1, \dots, \sigma_{d+1}, \mu_{d+3}, \dots,
\mu_{(d+ 1)(d+ 3)}\}) = \{\sigma_1, \dots, \sigma_{d+1}, \mu_{d+3}, \dots, \mu_{(d+1)(d+3)}\}$. Since $T_0$ and
$T_1$ are the only trees which contain $\{\sigma_1, \dots,$ $\sigma_{d+1}, \mu_{d+3}, \dots, \mu_{(d+1)(d+3)}\}$,
it follows that $\bar{\beta}(T_1) = T_1$. Inductively, we get $\bar{\beta}(T_i) = T_i$ for all $i$. Since
$\bar{\beta} $ is an automorphism of $\Lambda({\mathcal N}^{\hspace{.1mm}d+1}_{n})$, $\bar{\beta}$ is an
automorphism of $T_i$ for all $i$. This implies that $\bar{\beta}(\alpha_{\lfloor (d+1)/2\rfloor,i}) =
\alpha_{\lfloor (d+1)/2\rfloor,i}$ or $\alpha_{\lceil (d+1)/2\rceil,i}$ for all $i$. Now, either $\bar{\beta}$ is
identity on $T_0$ or $\bar{\beta}(\sigma_j) = \mu_{j(d+3)}$ for $0\leq j\leq d+1$. In the second case,
$\bar{\beta}(\alpha_{\lfloor (d+1)/2\rfloor, 2}) = \alpha_{\lceil (d+1)/2\rceil,
2(d+3)}$. This is not possible
since $\bar{\beta}(\alpha_{\lfloor (d+1)/2 \rfloor, 2}) = 
\alpha_{\lfloor (d+1)/2\rfloor, 2}$ or $\alpha_{\lceil
(d+1)/2\rceil, 2}$. Therefore, $\bar{\beta}|_{T_0}$ is the identity. Now, by the similar argument as in the case
${\mathcal M}^{\hspace{.1mm}d+1}_n$ it follows that $\beta$ is the identity in ${\rm Aut}({\mathcal
N}^{\hspace{.1mm}d +1}_n)$ and ${\rm Aut}({\mathcal N}^{\hspace{.1mm}d+1}_n) = \langle\psi\rangle$. Thus, ${\rm
Aut}(M^{\hspace{.1mm}d}_{n}) \supseteq {\rm Aut}({\mathcal M}^{\hspace{.1mm}d+1}_{n}) = \langle\psi\rangle = {\rm
Aut}({\mathcal N}^{\hspace{.1mm}d+1}_{n}) \subseteq {\rm Aut}(N^{\hspace{.1mm}d}_{n})$.

If $d\geq 4$ then the result follows from Corollary \ref{cor:auto}. For $d=3$, the result follows from Lemma
\ref{lemma:aut(M^3_29)}. So, assume that $d=2$. Using simpcomp \cite{simpcomp}, we found that ${\rm
Aut}(M^{\hspace{.1mm}2}_{19}) \cong \mathbb{Z}_{19} \cong {\rm Aut}(N^{\hspace{.1mm}2}_{19})$. The result follows
from this for $d=2$.
\end{proof}

\begin{remark}\label{remark}
{\rm Observe that a tree in ${\mathcal T}_1$ (in Example \ref{example:Gd}) is non-isomorphic to a tree in
${\mathcal T}_2$ (in Example \ref{example:Hd}). Since a tree in ${\mathcal T}_1$ (resp., ${\mathcal T}_2$) is
isomorphic to the dual graph of the link of the corresponding vertex of ${\mathcal M}^{\hspace{.1mm}d+1}_n$
(resp., ${\mathcal N}^{\hspace{.1mm}d+1}_n$), it follows that ${\mathcal M}^{\hspace{.1mm}d+1}_n$ is
non-isomorphic to ${\mathcal N}^{\hspace{.1mm}d+1}_n$ for all $d\geq 2$. However, the dual graphs of  ${\mathcal
M}^{\hspace{.1mm}d+1}_n$ and ${\mathcal N}^{\hspace{.1mm}d+1}_n$ are isomorphic (both are isomorphic to the graph
$G^d$). Since ${\mathcal M}^{\hspace{.1mm}d+1}_n$ and ${\mathcal N}^{\hspace{.1mm}d+1}_n$ are non-isomorphic, by
Proposition \ref{prop:bagchidatta}, $M^{\hspace{.1mm}d}_n$ and $N^{\hspace{.1mm}d}_{n}$ are non-isomorphic for
$d\geq 4$. Notice that $M^{\hspace{.1mm}2}_{29}$ has edges (namely, $a_ia_{i+12}$) which are in seven facets.
But, $N^{\hspace{.1mm}3}_{29}$ does not have any such edge. Thus, $M^{\hspace{.1mm}3}_{29}$ and
$N^{\hspace{.1mm}3}_{29}$ are non-isomorphic. Using simpcomp \cite{simpcomp}, we found that
$M^{\hspace{.1mm}2}_{19}$ and $N^{\hspace{.1mm}2}_{19}$ are non-isomorphic. }
\end{remark}

\section{Proof of Theorem \ref{thm:main}} \label{sec:proof}

For $d\geq 2$ and $m\geq 2$, let $D^{d+1}_{m+d+1}$ be the stacked $(d+1)$-ball with vertex-set $\{1, 2, \dots,
m+d+1\}$ and facets $\{k, k+1, \dots, k+d+1\}$, $1\leq k\leq m$. Let $M = \partial D^{d+1}_{m+d+1}$, $A$ be the
$d$-simplex $\{1, 2, \dots, d+1\}$ and $B$ be the $d$-simplex $\{m+1, m+2, \dots, m + d+1\}$. Then $M^{\prime}:=
M \setminus \{A, B\}$ triangulates $I \times S^{\hspace{.2mm}d -1}$. Recall that, a bijection $\psi\colon A \to
B$ is called {\em admissible} if for each vertex $u\in A$ there does not exist $v\in V(M)$ such that both
$\{u,v\}$ and $\{\psi(u),v\}$ are edges in $M$ (see \cite{bd8}). If $m\leq 2d+2$ then $\{d+1, d+1+ \lceil
\frac{m}{2} \rceil\}$ and $\{d+1+\lceil \frac{m}{2} \rceil, j\}$ are edges in $M$ for $d+1+\lceil \frac{m}{2}
\rceil \neq j\in B$. Thus, existence of an admissible map implies that $m\geq 2d+3$.  On the other hand, if $m
\geq 3d+3$ then there is no common neighbour of $i$ and $j$ in $M$ for $i\in A$, $j\in B$ and hence any bijection
$\psi\colon A\to B$ is admissible. Let $\sigma$ be a permutation on the set $\{1, \dots, d+1\}$ (i.e., $\sigma
\in {\rm Sym}(d+1)$). Consider the bijection $\varphi_{\sigma} \colon A \to B$ given by $\varphi_{\sigma}(i) =
m+\sigma(i)$. Consider the quotient complexes $Y := D^{d+1}_{m+d+ 1}/ \varphi_{\sigma}$ and $X^d_m(\sigma) :=
M^{\prime}/ \varphi_{\sigma}$. Then $\partial Y = X^d_m(\sigma)$. If $\varphi_{\sigma}$ is admissible then
$X^d_m(\sigma) \in {\mathcal K}(d)$ and triangulates an $S^{\hspace{.2mm}d -1}$-bundle over $S^1$. The case when
$md$ is even of the following lemma was proved in Lemma 3.3 of \cite{bd8}.

\begin{lemma}\label{lemma:sphere-bundle}
For $d\geq 2$, let $X^d_m(\sigma)$ be as above, where $\varphi_{\sigma}$ is admissible. Then $X^d_m(\sigma)$ is
orientable if and only if either $md$ is even and $\sigma$ is an even permutation or $md$ is odd and $\sigma$ is
an odd permutation. $($In particular, $X^d_m({\rm Id})$ is orientable when $md$ is even and non-orientable when
$md$ is odd.$)$
\end{lemma}

\begin{proof}
For $1 \leq k \leq m$, $1\leq l \leq d$, let $\delta_{k, l} $ denote the $d$-simplex $\{k, k + 1, \dots, k + d +
1\} \setminus \{k + l\}$ of $M$. Since $|M^{\prime}|$ is homeomorphic to $[0, 1]\times |\partial B|$,
$M^{\prime}$ is orientable. Observe that the following defines an orientation on $|M^{\prime}|$. (Here $\partial
B = S^{\hspace{.2mm}d -1}_{d + 1}(B)$.)
\begin{eqnarray} \label{eq:+orientation}
+ \delta_{k, l} & = & (-1)^{kd + l+1} \langle k, k + 1, \dots, k + l - 1, k + l + 1, \dots, k + d + 1\rangle.
\end{eqnarray}
(To check that (\ref{eq:+orientation}) defines a coherent orientation, one can take any orientation on
$(d-1)$-simplices of $M^{\prime}$. In particular, one can take positively oriented $(d-1)$-simplices as given in
(\ref{eq:+orientation2}) below.)

We can choose an orientation on $|\partial B|$ so that the orientation on $|M^{\prime}|$ as the product $[0, 1]
\times |\partial B|$ is the same as the orientation given in (\ref{eq:+orientation}). This also induces an
orientation on $|\partial A|$. Let $S_B$ (resp., $S_A$) denote the oriented sphere $|\partial B|$ (resp.,
$|\partial A|$) with this orientation. Now, as the boundary of an oriented manifold, $\partial |M^{\prime}| = S_A
\cup (-S_B)$ (cf. \cite[pages 371--372]{td}). Therefore, $|M^{\prime}/\varphi_{\sigma}|=
|M^{\prime}|/|\varphi_{\sigma}|$ is orientable if and only if $|\varphi_{\sigma}| \colon S_A \to S_B$ is
orientation preserving (cf. \cite[pages 134--135]{st}). (Here, $|\varphi_{\sigma}| \colon |\partial A| \to
|\partial B|$ is the homeomorphism induced by $\varphi_{\sigma}$.)

Note that $(d-1)$-simplices of $M$ are $\delta_{k, i, j} = \{k, k+ 1, \dots, k + d + 1\} \setminus \{k + i, k +
j\}$, $0 \leq i < j \leq d + 1$, $(i, j) \neq (0, d+1)$, $1 \leq k \leq m$. Consider the orientation on the
$(d-1)$-skeleton of $M^{\prime}$ as\,:
\begin{eqnarray} \label{eq:+orientation2}
+ \delta_{k, i, j} &\!\!=\!\!& (-1)^{kd + i + j} \langle k, \dots, k + i - 1, k + i + 1, \dots, k + j - 1, k + j
+ 1, \dots\rangle.
\end{eqnarray}
Then $[\delta_{m, i+1}, \delta_{m, 0, i+1}] = -1$ (resp., $[\delta_{1, i}, \delta_{1, i, d+1}] = 1$) for $0\leq i
\leq d$. This implies that $|\partial B|$ (resp., $|\partial A|$) with orientation given in
(\ref{eq:+orientation2}) is $S_B$ (resp., $S_A$). (For $\beta = \{v_0, v_1, \dots, v_d\}\in M^{\prime}$ and
$\alpha = \{v_1, \dots, v_d\}\in\partial B$, if $[\beta, \alpha] = -1$ then $+\alpha = \langle v_1, \dots,
v_d\rangle$ with the orientation given in (\ref{eq:+orientation2}) $\Longleftrightarrow$ $+\beta = \langle v_1,
v_0, v_2, \dots, v_d \rangle$ with the orientation given in (\ref{eq:+orientation}) $\Longleftrightarrow$
$(\stackrel{\longrightarrow}{v_1v_0}$, $\stackrel{\longrightarrow}{v_1v_2}, \dots,
\stackrel{\longrightarrow}{v_1v_d})$ is the orientation of $|M^{\prime}|$ $\Longleftrightarrow$
$(\stackrel{\longrightarrow}{v_1v_2}, \dots, \stackrel{\longrightarrow}{v_1v_d})$ is the orientation of
$|\partial B|$ $\Longleftrightarrow$ $\langle v_1, v_2, \dots, v_d \rangle$ is positive in $S_B$.)

For $1\leq i\leq d+1$, consider the $(d-1)$-simplex $\delta_{1, i,d+1} = \{1, \dots, d+1\}\setminus\{i+1\}$ of
$\partial A$. Then $\varphi_{\rm Id}(\delta_{1,i,d+1}) = \{m+1, \dots, m+d+1\} \setminus \{m+i+1\}= \delta_{m,0,
i + 1}$. Therefore, from (\ref{eq:+orientation2}), $\varphi_{\rm Id}(+\delta_{1, i, d+1}) = (-1)^{md} \delta_{d,
0, i+ 1}$. Thus, $|\varphi_{\rm Id}| \colon S_A \to S_B$ is orientation preserving (resp., reversing) if $md$ is
even (resp., odd). Also $|\sigma| \colon S_A \to S_A$ is orientation preserving (resp., reversing) if $\sigma$ is
an even (resp., odd) permutation. Since $\varphi_{\sigma} = \varphi_{\rm Id} \circ \sigma$, it follows that
$|\varphi_{\sigma}| \colon S_A \to S_B$ is orientation preserving if and only if $md$ is even and $\sigma$ is an
even permutation or if $md$ is odd and $\sigma$ is an odd permutation. The lemma now follows.
\end{proof}

\begin{lemma}\label{lemma:orientability}
For $d\geq 2$ and $n=d^{\hspace{.3mm}2}+5d+5$, let $M^{\hspace{.2mm}d}_{n}$ and $N^{\hspace{.2mm}d}_{n}$ be as in
Examples $\ref{example:Md}$ and $\ref{example:Nd}$, respectively. Then $M^{\hspace{.2mm}d}_{n}$,
$N^{\hspace{.2mm}d}_{n}$ are orientable if $d$ is even and are non-orientable if $d$ is odd.
\end{lemma}

\begin{proof}
We present a proof for $M^{\hspace{.2mm}d}_{n}$. Similar arguments work for $N^{\hspace{.2mm}d}_{n}$. Let
${\mathcal M}^{\hspace{.2mm}d +1}_{n}$ be as in Example \ref{example:Md}. Let $E_1$ (resp., $E_2$) be the pure
$(d+1)$-dimensional subcomplex of ${\mathcal M}^{\hspace{.2mm}d +1}_{n}$ whose facets are $\sigma_0 \dots,
\sigma_{n -1}$ (resp., $\mu_0 \dots, \mu_{n-1}$). So, $\Lambda(E_i) = C_i$, $1\leq i\leq 2$.

Clearly, $E_1$ is isomorphic to the pseudomanifold $D^{d + 1}_{n + d + 1}/\varphi_{\rm Id}$, where
$D^{d+1}_{n+d+1}$ is the stacked $(d+1)$-ball defined at the beginning of this section. Thus, $\partial E_1$ is
isomorphic to $X^d_{n}(\rm Id)$. Therefore, $\partial E_1$ triangulates an $S^{\hspace{.2mm}d- 1}$-bundle over
$S^1$ and, by Lemma \ref{lemma:sphere-bundle}, is orientable if and only if $dn$ is even. Thus (since $n$ is
odd), $\partial E_1$ is orientable if and only if $d$ is even. So, if $d$ is odd then $\partial E_1$ is
non-orientable and hence (since $|M^{\hspace{.2mm}d}_{n}|$ can be obtained from $|\partial E_1|$ by attaching
handles) $M^{\hspace{.2mm}d}_{n}$ is non-orientable.

Again, the bijection $f \colon \mathbb{Z}_{n} \to V(E_2)$ given by $f(i) = a_{(d+3)i}$ defines an isomorphism
between $D^{d+1}_{n+d+ 1}/\varphi_{\rm Id}$ and $E_2$. Thus, $\partial E_2$ is isomorphic to $X^d_{n}(\rm Id)$.
Therefore, $\partial E_2$ triangulates an $S^{\hspace{.2mm}d- 1}$-bundle over $S^1$ and, by Lemma
\ref{lemma:sphere-bundle}, orientable if and only if $d$ is even.

For $0\leq i\leq n-1$, let $F_i$ be the stacked $(d+ 1)$-ball whose facets are $\alpha_{1, i}, \dots, \alpha_{d,
i}$. Then, ${\mathcal M}^{\hspace{.2mm}d +1}_{n} = E_1\cup E_2 \cup (\cup_{i= 0}^{n-1} F_i)$ and
$M^{\hspace{.2mm}d}_{n}$ is obtained from $\partial E_1 \cup \partial E_2$ by attaching handles $\partial
F_i\setminus \{A_i, B_i\}$, where $A_i = a_{i -d- 1} \cdots a_{i-2}a_{i}$ and $B_i = a_{i+d+2} a_{i+ 2(d+3)-1}
\cdots a_{i+ d(d +3)-1}a_{i}$, $0 \leq i\leq n-1$ (additions are modulo $n$).

Now, assume that $d$ is even. So, $\partial E_1$, $\partial E_2$, $\partial F_i$, $0\leq i\leq n -1$, are
orientable. Consider the orientation on $\partial E_1$, $\partial E_2$ and $\partial F_i$, $0\leq i\leq n-1$,
given by\,:
\begin{subequations}
\begin{align}
+\sigma_{i,l} =  & \, (-1)^l\langle a_{i-d-1}, \dots,
a_{i-d-2+l}, a_{i-d+l}, \dots a_{i}\rangle, \label{+E1}
\\
+\mu_{i,l} =  & \, (-1)^{l+1}\langle a_{i+(d+2)}, \dots, a_{i+l(d+3)-1},a_{i+(l+2)(d+3)-1}, \dots,
a_{i+(d+1)(d+3)-1},
a_i\rangle, \label{+E2}
\\
+\alpha_{k,i,l} =  & \, (-1)^l \langle b_{k,i,1}, \dots, b_{k,i,l}b_{k,i,l+2}, \cdots b_{k,i,d+2}\rangle,
\label{+Fi}
\end{align}
\end{subequations}
where $(b_{k,i,1}, b_{k,i,2}, \dots, b_{k,i,d+2}) = (a_{i-2-d+k}, \dots, a_{i-2}a_{i}a_{i+(d+2)} a_{i+2(d+3)-1},
\dots, a_{i+k(d +3) -1})$, for $1\leq k\leq d$, $0\leq l\leq d+1$. From the proof of Lemma
\ref{lemma:sphere-bundle}, (\ref{+E1}) (resp., (\ref{+E2})) defines an orientation on $\partial E_1$ (resp.,
$\partial E_2$). Also, (\ref{+Fi}) defines an orientation on $\partial F_i$, $0 \leq i\leq n-1$.

Observe that $A_i = \sigma_{i,d} = \alpha_{1,i,d+1}$ and $+\sigma_{i,d}= -\alpha_{1,i,d+1}$. Also, $B_i =
\mu_{i,d} = \alpha_{d,i,0}$ and $+\mu_{i,d} = -\alpha_{d,i,0}$. Now, let $\gamma$ be a $(d-1)$-face of $A_i$. Let
$\gamma_{E_1}$ (resp., $\gamma_{F_i}$) be the $d$-face of $\partial E_1$ (resp., $\partial F_i$) other than $A_i$
which contains $\gamma$. Then (with any orientation of $\gamma$)
\begin{eqnarray}
[\gamma_{E_1}, \gamma] = -[\sigma_{i,d}, \gamma] = [\alpha_{1, i,d+1}, \gamma] = - [\gamma_{F_i}, \gamma].
\label{E1&Fi}
\end{eqnarray}
Similarly, if $\beta$ is a $(d-1)$-face of $B_i$ and $\beta_{E_2}$ (resp., $\beta_{F_i}$) is the $d$-face of
$\partial E_2$ (resp., $\partial F_i$) other than $B_i$ which contains $\beta$. Then (with any orientation of
$\beta$)
\begin{eqnarray}
[\beta_{E_2}, \beta] = -[\mu_{i,d}, \beta] = [\alpha_{d,i,0}, \beta] = - [\beta_{F_i}, \beta]. \label{E2&Fi}
\end{eqnarray}
Since
$$
M^{\hspace{.1mm}d}_{n} = (\partial E_1 \setminus \{A_0, \dots, A_{n-1}\}) \cup (\partial E_2 \setminus \{B_0,
\dots, B_{n-1}\}) \cup ({\displaystyle \bigcup_{i=0}^{n-1}} (\partial F_i\setminus\{A_i, B_i\})),
$$
it follows from (\ref{E1&Fi}) and (\ref{E2&Fi}) that the orientations defined by (\ref{+E1}), (\ref{+E2}) and
(\ref{+Fi}) give a coherent orientation on $M^{\hspace{.1mm}d}_{n}$. Thus, $M^{\hspace{.1mm}d}_{n}$ is
orientable. This completes the proof.
\end{proof}

We now in a position to prove our main result.

\begin{proof}[{\it Proof of Theorem \ref{thm:main}}]
Let $d\geq 2$ and $n = d^{\hspace{.3mm}2}+5d+5$. \begin{description} \item[{\rm Part (a):}] From Lemmata
\ref{lemma:Md} and \ref{lemma:Nd}, ${\mathcal M}^{\hspace{.1mm}d+1}_{n}$ and ${\mathcal N}^{\hspace{.1mm}d
+1}_{n}$ are neighborly members of $\overline{\mathcal K}(d+1)$. This implies (by Proposition
\ref{prop:bagchidatta} and (\ref{eq:skel})) that $M^{\hspace{.1mm}d}_{n}$ and $N^{\hspace{.1mm}d}_{n}$ are
neighborly members of $\Kd$. Also, $M^{\hspace{.1mm}d}_{n}$ and $N^{\hspace{.1mm}d}_{n}$ are non-isomorphic (see
Remark \ref{remark}).

\item[{\rm Part (b):}] If $d\neq 3$ then tightness follows from Proposition \ref{prop:ef}. Since
$M^{\hspace{.1mm}3}_{29} = \partial {\mathcal M}^{\hspace{.1mm}4}_{29}$ and $N^{\hspace{.1mm}3}_{29} = \partial
{\mathcal N}^{\hspace{.1mm}4}_{29}$, tightness follows from Proposition \ref{prop:tight3} for $d=3$.

\item[{\rm Part (d):}] If $d\geq 4$ then the result follows from
Corollary \ref{prop:lss}.
If $d=3$ then the result follows from Proposition \ref{prop:tight3}.

\item[{\rm Part (c):}] The result follows from Part (d).

\item[{\rm Part (e):}] If $d\geq 4$ then the result follows from Proposition \ref{prop:kalai} and Lemma
\ref{lemma:orientability}.  Since $M^{\hspace{.1mm}2}_{19}$ and $N^{\hspace{.1mm}2}_{19}$ are orientable (by
Lemma \ref{lemma:orientability}) and neighborly, $\beta_1(M^{\hspace{.1mm}2}_{19}; \mathbb{Z}) =
\beta_1(N^{\hspace{.1mm}2}_{19}; \mathbb{Z}) = 2- (19 - \binom{19}{2} + \frac{2}{3}\binom{19}{2}) = 40$ and hence
$M^{\hspace{.1mm}2}_{19}$ and $N^{\hspace{.1mm}2}_{19}$ both triangulate $(S^1 \times S^1)^{\#20}$. This proves
the result for $d=2$. If $d=3$ then the result follows from  Lemmata \ref{lemma:|(M^3_29|} and 
\ref{lemma:orientability}.

\item[{\rm Part (f):}] By Part (b) and Proposition \ref{prop:minimal},
$M^{\hspace{.1mm}d}_{n}$ and $N^{\hspace{.1mm}d}_{n}$ are strongly minimal
for $d\geq 2$.

\item[{\rm Part (g):}] The result follows from Lemma \ref{lemma:aut(MdNd1)}.
\end{description}
This completes the proof of the theorem.
\end{proof}


\section{Summary: Known neighborly members of \boldmath{${\mathcal K}(d)$}}

Any triangulated 2-manifold is a member of ${\mathcal K}(2)$. In Table \ref{tbl:tbl1}, we
summarize the known and some open cases for neighborly members of Walkup's class ${\mathcal K}(d)$ for $d\geq 3$.


\begin{table}[h]
\centering {\small
\begin{tabular}{|c|c|c|c|c|c|}
\hline
&&&&&\\[-3mm]
$\beta_1(K)$ & $d$ & $n$ & $K$ & $|K|$ & References\\[1mm]
\hline
&&&&&\\[-3mm]
0 & $d$ & $d+2$ & $S^{\hspace{.1mm}d}_{d+2}$ &
$S^{\hspace{.1mm}d}$ & \\[1mm]
1 & $d$ even & $2d+3$ & $K^{d}_{2d+3}$ &
$S^{d-1}\times S^1$ & \cite{ku86} \\[1mm]
1 & $d$ odd & $2d+3$ & $K^d_{2d+3}$ & $\TPSSD$ &
\cite{ku86} \\[1mm]
2 & $d\geq 4$ & $-$  & Not possible & &\cite{si1} \\[1mm]
3 & 4 & 15 & $M^4_{15}$ & $(\TPSS)^{\#3}$ & \cite{bd10} \\[1mm]
3 & 4 & 15 & $N^4_{15}$ & $(S^3\times S^1)^{\#3}$ & \cite{si2} \\[1mm]
5 & 5 & 21 & ? & &\cite{ef}\\[1mm]

7 & 4 & 20 & ? & &\cite{ku95}\\

8 & 4 & 21 & $M^4_{21}$ & $(S^3\times S^1)^{\#8}$ &\cite{bdns1} \\[1mm]

8 & 4 & 21 & $N^4_{21}$ & $(\TPSS)^{\#8}$ &\cite{bdns1} \\[1mm]

14 & 4 & 26 & $N^4_{26}$ & $(\TPSS)^{\#14}$ &\cite{bdns1} \\[1mm]

$d^2+5d+6$ & $d$ even & $d^2+5d+5$ & $M^{d}_{d^2+5d+5}$ &
$(S^{d-1}\times S^1)^{\#\beta_1}$ & this paper\\[1mm]

$d^2+5d+6$ & $d$ odd & $d^2+5d+5$ & $M^d_{d^2+5d+5}$ &
$(\TPSSD)^{\#\beta_1}$ & this paper\\[1mm]
\hline
\end{tabular}
} \caption{Known and some open cases for neighborly members of ${\mathcal K}(d)$, $d\geq 3$} \label{tbl:tbl1}
\end{table}

\bigskip

\noindent {\bf Acknowledgement:}
This work is supported in part by UGC Centre for Advanced Studies.
The second author thanks IISc Mathematics Initiative for support.
The authors thank Bhaskar Bagchi and the anonymous
referees for many useful comments.



\end{document}